\documentclass[11pt,oneside]{article}
\usepackage{authblk}
\usepackage{amsthm, amsmath, amssymb, bbm}
\usepackage{graphicx, wrapfig, caption, subcaption}
\graphicspath{ {./img/} }
\allowdisplaybreaks

\usepackage[dvipsnames]{xcolor}

\def\E{\mathbb{E}}
\def\R{\mathbb{R}}
\def\N{\mathbb{N}}
\def\Q{\mathbb{Q}}

\def\P{\mathbb{P}}
\def\llfloor{\left\lfloor}
\def\rrfloor{\right\rfloor}
\newcommand*\dd{\mathop{}\!\mathrm{d}}
\allowdisplaybreaks

\newtheorem{theorem}{Theorem}[section]
\newtheorem{definition}{Definition}[section]

\newtheorem{remark}{Remark}[section]
\newtheorem{lemma}{Lemma}[section]

\title{Stability in stochastic hypergraph matching III: general reneging}
\author{Doan Dai Nguyen, Ana Bu\v{s}i\'{c}}
\affil{Inria and DI ENS, École Normale Supérieure\\PSL University, Paris, France}

\begin{document}
	\maketitle
	\begin{abstract}
		In many real-life matching problems, waiting agents might abandon before being matched, such as patients deceasing before receiving organs, passengers/drivers cancelling ride requests, or raw materials/intermediary products degrading in production lines. This poses the need for incorporating reneging in stochastic matching models.
		
		In this work, we consider matching models on hypergraphs with batch arrivals and general-weight matchings. Since our model allows fractional weights, we may not be able to talk about individual items, and thus reneging is not required to be independent between items of the same class. For the simplicity sake's, we assume items arrive at discrete time.
		
		We show that stability depends on the exact nature of reneging, in stark contrast with the non-reneging case where it depends on the arrivals only through the arrival rates. To our best knowledge, this is the first such sensitivity result in stochastic matching.
		
		We uncover a new stabilising mechanism which exists neither in the non-reneging case nor in the graph case, which explains why incorporating reneging in stochastic matching is not straightforward. Together with balancing mechanism as hinted in the online assignment framework for the non-reneging case, it gives a criterion necessary and sufficient for stability.
		
		Finally, whilst verifying stability is a hard problem, we give a family of MaxWeight-type policies parameterised by $\varepsilon > 0$, which are maximally stabilising for all $\varepsilon$ sufficiently small. Unfortunately there is no effective bound for $\varepsilon$, but we show how to adjust its value during implementation.
	\end{abstract}
	
	\section{Introduction}

This article is the third in a series on necessary and sufficient criteria for the problem of stochastic matching on hypergraphs. In the first article \cite{Nguyen2026}, we have studied the problem in the classical setting, characterised by a hypergraph $G = (V, E)$ and a distribution $\mu$ on $V$. At time $t \in \N$, a class $v \in V$ is drawn i.i.d. according to $\mu$, and a class-$v$ item arrives in the buffer. The arrived items then may be combined according to some hyperedge $e \in E$, which we call a matching.

This model was generalised in the second article \cite{Nguyen2026a}, where multiple items of multiple classes may arrive at the same time, which can happen at any moment $t \in \R_{\geq 0}$. Moreover, a matching now may require multiple items of a given class; the weights of a class $v$ for a matching $e$ may be negative, which we interpret as each type-$e$ matching generates new class-$v$ items. Negative weights and continuous time allow describing a richer class of systems, including assemble-to-order and production line with ingredients coming in both as discrete batches and continuous flows.

This article extends the model to account for reneging: items in the buffer may abandon the system either while waiting to be matched.

\subsection{Motivation}

Such phenomenon arises in many applications context, such as organ donation, production lines, or home swaps.

In kidney multi-way exchanges, patient-donor pairs may leave the queue before being matched, but they might also leave \textit{after} being matched but before the exchange takes place. Such cancellations may result from medical or psychological reasons, voluntary withdrawal, or the availability of alternative transplantation options \cite{DeKlerk2010, Verissimo2022}.

In assemble-to-order systems, there might be impatient demand and/or perishable goods. In this setting, a matching can be interpreted as a collection of raw materials required to produce a batch of goods together with the customers who are to receive that batch. If any of the raw materials perishes before the matching is realised, for example due to the time it takes to assemble the materials and/or to process, the batch cannot be produced. Similarly, if any customer abandons their demand after being matched, the remaining customers may no longer be able to receive the batch. Such interactions between perishable inventory and customer impatience have been studied in production and inventory systems \cite{Ioannidis2013}.

In ride-pooling, several passengers may be jointly matched to a vehicle. Such a match may cease to be feasible before the ride takes place if one of the passengers cancels or abandons the service either before or after being matched, for instance because of excessive waiting or pickup times \cite{Bujak2026}. The departure of one passenger may in turn invalidate the match for the remaining passengers, since the pooled ride may no longer satisfy the relevant routing or cost constraints.

In barter exchange, multiple parties may agree to exchange their goods with one another, with the exchange taking place only if all parties remain willing to participate. Consequently, the exchange as a whole may fail if one or more participants subsequently withdraw from the agreement \cite{Gupta2019}.

As in many matching models, the primary goal in designing matching policies is to ensure stability, in the sense that no agents wait indefinitely. While stability in the classical setting is well understood, the possibility of reneging introduces an additional challenge: a matching may fail when any one of its constituent agents abandons the system. Consequently, even agents that do not renege may experience increased waiting times, and it is therefore unclear how stability can be guaranteed.

\subsection{Literature review}

In contrast with queueing theory where reneging has been well-analysed \cite{Singh2022}, this topic remains comparatively understudied in stochastic matching theory. As remarked by Aveklouris et al. \cite{Aveklouris2024}, at least in stochastic matching on bipartite graphs, there are much fewer works with reneging, most of which assume reneging times to be either deterministic or exponential distributed (see Aveklouris et al. \cite{Aveklouris2024} and the references therein). Only very few assume more general reneging distributions, which studied fluid and diffusion limits \cite{Aveklouris2024, Ding2021}, leaving the question of stability criteria unanswered.

To ensure stability without a complete characterisation, many existing works studied some small specific bipartite graphs, notably $K_{1, 1}$ \cite{Buke2017, Kohlenberg2025, Boxma2011, Liu2019}, which corresponds to double-ended queues, $K_{1, 2}$ \cite{Kanoria2017}, $K_{2, 2}$ \cite{Afeche2014}, or $N$-system \cite{Castro2020}, on which exact computations, such as finding stationary distributions, are feasible. Other works on general bipartite matching models assumed that all classes renege \cite{Blanchet2022, Zenios1999, Aveklouris2024} or focused on finite horizon \cite{Khademi2021}.

Jonckheere et al. \cite{Jonckheere2023} later gave a necessary and sufficient criterion for stability in stochastic matching on general graphs, where some classes may not renege. Under exponential reneging times, their condition \cite[Theorem 3.2]{Jonckheere2023} resembles and generalises \texttt{NCOND} in the non-reneging case \cite[Theorem 2]{Mairesse2016}. However, stability in the non-reneging case is also characterised linear algebraically \cite[Proposition 7 and 8]{Comte2021}. Finding such a criterion for the reneging case, as well as establishing stability under more general distributions for reneging times, remains an open problem even for stochastic matching on graphs.

Beyond graphs, Zubeldia et al. characterised stability for a particular example of stochastic matching on hypergraphs with exponential reneging time \cite[Theorem 1]{Zubeldia2026}. Notably, one of the case they considered, called $Y$-topology, is a matching model of one unique hyperedge with three classes, one of which is non-reneging and the other two are reneging.

\subsection{Contributions and organisation}

The goal of this work is to bridge this gap we have identified above. Specifically, we consider the weighted stochastic matching model with some (possibly all) item classes being impatient, meaning that they will leave the queue once their waiting time runs out. Regarding the reneging time distribution, we only require that if a class is reneging, the reneging time has finite mean.

Then, we extend the framework of assignment and reassignment to our new setting. Whilst much of the proof carry over, the generalisation is not direct. One of the main difficulties is that with reneging, the drift of Lyapunov function from the non-reneging case \cite[Function $L_+$]{Nguyen2026a} depends on the queue length and not just the assignment/reassignment rate. Exploring the reasons in depth leads to the conclusion that the idea of assignment and reassignment alone is not sufficient, and one must think about \textit{how} the matchings amongst items assigned to the same matching type should be done in order to guarantee stability.

To this end, we introduce a new technique, whereby we enlarge the buffer state description with a matching process, which is essentially a counter acting as a proxy to the number of matchings available. Before, the dynamics of this process is automatically given by the queue-length process under the assumption that we always make as many matchings as possible using items assigned to the same matching type; now, we require the control of this process manually at all times. In exchange for a somewhat more elaborate system description, we have a better control of the dynamics, which eventually leads to necessary and sufficient criteria generalising those for the non-reneging case, similar to how in the case of graphs, the criterion given by Jonckheere et al. \cite[Theorem 3.2]{Jonckheere2023} generalises \texttt{NCOND} \cite[Theorem 2]{Mairesse2016}. Consequently, much of the proofs are also different: in particular, for the proof of sufficiency, new Lyapunov function involving the matching process is designed.

The outline of the paper is as follow:
\begin{itemize}
	\item In Section \ref{section: preliminaries}, we formally define our model. As in the previous paper, we allow batch arrivals with some mild assumption. For simplicity, we assume that the weights are non-negative and that the discrete-event setting, where the events - arrivals, matchings, abandonments - only occur at discrete time $t \in \N$.
	
	\item In Section \ref{section: a case study}, we review the case of $Y$-topology as studied by Zubeldia et al. \cite{Zubeldia2026}, where we characterise the stability region. This reveals a new stabilising mechanism which does not exist in the non-reneging case or with graphs, and explains why generalising the result in the graph case by Jonckheere et al. \cite{Jonckheere2023} is not straightforward, even with the online assignment framework.
	
	\item In Section \ref{section: necessity}, we combine the insight from the case study with the mechanism of managing imbalance within hyperedges as presented in the non-reneging case \cite[Condition (1)]{Nguyen2026a}, to devise a stability criterion for the reneging case (Condition \eqref{eq: ncond}). We also show that it is necessary for stability (Lemma \ref{lemma: necessity of ncond}).
	
	\item However, in trying to show the sufficiency of Condition \eqref{eq: ncond} by generalising the proof from the case study in Section \ref{section: a case study}, we encounter a term in the Lyapunov function, whose drift is not easily analysable.
	
	To this end, we amend the online assignment framework by introducing the notion of matching process, which serves as a proxy for this difficult term, and whose drift is more amenable to analysis.
	
	\item Once the notion is developed, Condition \eqref{eq: ncond} is shown to be sufficient for stability (Lemma \ref{lemma: sufficiency of ncond}), which is the subject of Section \ref{section: sufficiency}.
	
	As a corollary, we obtain a family of MaxWeight-type policies indexed by a parameter $\varepsilon > 0$, which are maximally stable for all $\varepsilon$ sufficiently small. Unfortunately, a priori there is no effective upper bound, but we outline strategies to adjust $\varepsilon$ as a policy is implemented.
	
	\item Using the notion of generalised matching types \cite[Section 6]{Nguyen2026a}, we generalise our stability criterion to matching models of arbitrary weights (Theorem \ref{theorem: necessity and sufficiency of ncond for general-weight case}).
	
	\item  Finally, Section \ref{section: conclusion} concludes with some remarks.
\end{itemize}
	\section{Preliminaries}
\label{section: preliminaries}

\subsection{Matching models}

We recall the setting from the previous paper \cite{Nguyen2026a}, where we have a weighted matching hypergraph $G = (V, E)$ on a finite set $V$ of classes. The hyperedges $e \in E$ are the matching types with weights: specifically, the weight of vertex $u$ in the hyperedge $e$ is denoted by $e_u$. Unless specified otherwise, we assume hereinafter that $e_u \geq 0$, and we write $u \in e$ whenever $e_u > 0$ to mean that the hyperedge $e$ contains the vertex $u$. Denote $|e| = |\{u \mid u \in e\}|$ and $\|e\|_1 = \sum_{u \in e} e_u$. 

By abuse of notation, we see $G$ both as a subset of $V \times E$, which allows us to write $(u, e) \in G$ whenever $u \in e$ and $e \in E$, and also as a matrix in $\R^{V \times E}$ where the entry $G_{u, e}$ is the weight $e_u$. As with $e$, we denote by $|G|$ the size of $G$ as a subset of $V \times E$, and $\|G\|_1$ its $\ell^1$-norm as a matrix in $\R^{V \times E}$.

To this, we add a discrete-time arrival process characterised by a distribution $\mu$ on $\R_{\geq 0}^V \setminus \{0\}$ and a buffer $Z = (Z_v(t))_{v \in V}$, where $Z_v(t)$ denotes the number of class-$v$ items in the buffer at time $t$, waiting to be matched. For convenience, let
\[
	\mu_v = \int_{\R^V} k_v \dd \mu(k),
\]
as well as
\[
	\| \mu\|_1 = \sum_{v \in V} \mu_u = \int_{\R^V} \|k\| \dd \mu(k),
\]
which we assume to be finite. Note that by definition, $\mu_v \geq 0$, so this implies that $\mu_v$ is also finite for all $v \in V$. Similar to the previous paper, we also denote
\[
	\mu_v^2 = \int_{\R^V} k_v^2 \dd \mu(k),
\]
and
\[
	\|\mu\|_2^2 = \sum_{v \in V} \mu_v^2 = \int_{\R^V} \| k \|^2 \dd \mu(k),
\]
which we also assume to be finite.

In addition to this, we also have reneging for each class $v \in V$. More specifically, we have $R = \left(R_v (t)\right)_{v \in V, t \in \N}$ a sequence of independent random variables on $[0, 1]^V$ of mean $\E[R(t)] = \alpha$.

\begin{remark}
	Note that the random variables $R_v(t)$ for $t \in \N$ are \emph{not} necessarily identically distributed. As we will see in Remark \ref{remark: independent reneging} below, even in the simplest setting, this is not the case.
\end{remark}

A realisation of the matching model is as follow: we start with some initial buffer $Z(0)$. At time $t$, an arriving batch $k = k(t)$ is drawn i.i.d. according to $\mu$, and for each $v$, there are $k_v$ class-$v$ items arriving. The current quantity of class-$v$ items is now $Z_v(t) + k_v$, amongst which $R_v(t) (Z_v(t) + k_v)$ abandon the queue without being matched. Thus on average, the percentage of reneging class-$v$ item at each period is $\alpha_v$, and in particular, the case $\alpha_v = 0$ corresponds to class-$v$ items not reneging.

In particular, if items of some class $v$ reneging, then it is necessary that there are newly arriving items, for otherwise, the number of items will eventually be too small for any matchings of type $e \ni v$ involving class-$v$ items to occur. We thus assume that if $\alpha_v > 0$ then $\mu_v > 0$.

The buffer state before matchings but after reneging is given by $B(t) = (B_v(t))_{v \in V}$ where $B_v(t) = \left(1 - R_v(t)\right) (Z_v(t) + k_v)$.

A collection of matchings $m$ can be expressed as a linear combination of hyperedges $m = \sum_{e \in E} m_e e$, where the number of class-$v$ items used is given by $m_v = \sum_{v \in V} m_e e_v$. We only consider integral matchings, and in particular we require that $m \in \N^E$; note that $m = 0$ corresponds to doing nothing, which is also allowed even when there are feasible matchings. Such a collection $m$ is feasible if the buffer may allow it, meaning we have $m \leq B(t)$. If it is realised, then the resulting buffer state at time $t + 1$ is given by
\[
	Z(t+1) = B(t) - m.
\]

\begin{remark}
	\label{remark: independent reneging}
	To understand the motivation behind our choice of modelling reneging, it is fruitful to revisit the classical setting where only one items arrive at at time and the matching model is unweighted. In this case, the number $Z_v(t)$ of class-$v$ items in the buffer at time $t$ is an integer, and one may talk about individual item as an unit.
	
	Assuming the reneging time to be exponentially distributed with parameter $\frac{1}{\alpha_v}$ - or, what amounts to the same upon rounding, to be geometrically distributed with parameter $\alpha_v$ - at time $t$, each class-$v$ item abandons with probability $\alpha_v$ independently and identically. $R_v(t)$ simply denotes the percentage of abandoning class-$v$ items; in this case, $R_v(t)$ follows the normalised binomial distribution $\frac{1}{Z_v(t) + k_v} B(Z_v(t) + k_v, \alpha_v)$.
	
	As such, our modelling of reneging generalises that in the unweighted case. The need for this stems from the fact that since the quantity of class-$v$ items can be non-integer, we may no longer talk about them as units.
\end{remark}


\subsection{Matching policies}

It may occur that for a given buffer state after reneging but before matching $B(t)$, there can be multiple collection of feasible matchings possible, and we have to choose one by specifying a matching policy $\Phi = (\Phi_t)_{t \in \N}$. Naturally, we require $\Phi$ to be non-anticipative, so in particular, it may rely only on the information available at time $t$, which are the arrival thus far, the decisions it has made up to time $t-1$, and possibly some independent source of randomness for randomised policies.

Formally speaking, let $U = (U(t))$ be a sequence of independent random variables, and $\mathcal{H}_t = \sigma\left(\left\{k(t'), U(t') \mid t' \leq t\right\}\right)$ be the Borel $\sigma$-algebra generated by the arrivals $(k(t'))_{t' \leq t}$ up to time $t$ and $(U(t'))_{t' \leq t}$, then $\mathcal{H}$ is a filtration, wither we require $\Phi$ to be adapted.

We say that $\Phi_t$ to be Markovian if it depends only on $Z(t)$, $k(t)$, and $U(t)$. If furthermore, $U(t)$'s are identically distributed, then $\Phi_t$ has no notion of time $t$, and we say that $\Phi$ is stationary.

Regardless of the properties of $\Phi$, as in the usual setting, the goal in designing $\Phi$ is to ensure stability. A policy $\Phi$ induces a queue-length process $Z$. Following the previous paper, we say that $\Phi$ stabilises $(G, \mu, R)$ if at any point, the expected time before the buffer becomes ``small'' is finite. Formally speaking, this means there exists a compact set $C \subset \R^V_{\geq 0}$ such that for all $T \geq 0$, we have
\[
	\E\left[\min\{t \geq 0 \mid Z(T + t) \in C\} \mid \mathcal{H}_T\right] \leq \infty.
\]

\begin{remark}
	Compared with the previous paper, we ignore the order of arrivals and implicitly assume that if items of some class are used, they will be chosen in the order of arrival. Similar to the remark we have made upon defining continuous-time matching models \cite[Remark 8.1]{Nguyen2026a}, the notion of stability and all the analysis hereinafter use only the queue-length process $Z$ and not the order of arrivals.
	
	Every matching policy induces such a process, even though it needs not be a characterisation, meaning that multiple polices may induce the same process. However, as we will see, there are size-based matching policies of maximum stability region, so this does not incur any loss of generality.
\end{remark}

\subsection{Supports, assignment, assignment rates, and online assignment policies}

The definition of supports, assignment, and assignment rates remain the same, which we recall for the sake of convenience. However, due to reneging, the behaviour of online assignment polices is a bit different, which we define rigorously for the sake of clarity. 

\begin{definition}
	\label{definition: G-support}
	A set $U \subset G$ is called a $G$-support if for all $e \in E$, there exists a vertex $v$ such that $(v, e) \in G \setminus U$. In particular, $U$ does not contain any hyperedge $e$ entirely.
\end{definition}

\begin{definition}
	By assigning some class-$v$ items to a matching type $e \ni v$, we mean that these items then can only be matched by a type-$e$ matching. In particular, even if there are other available matching types using class $v$, these items will still not be matched.
\end{definition}

\begin{definition}
	Given an arrival rate $\mu$, we call a tuple $\nu = (\nu_{v, e})_{(v, e) \in G} \in \R_{\geq 0}^G$ an assignment rate if for all $v \in V$, we have $\sum_{e \ni v} \nu_{v, e} = \mu_v$.
\end{definition}

To describe the behaviour of an online assignment policy, we first enlarge the buffer into $X = (X_{v, e}(t))_{(v, e) \in G}$, with one queue $X_{v, e}$ for each $(v, e) \in G$. Moreover, for each $e \ni v$ and $t \in \N$, we have a random variable $R_{v, e}(t)$ independent from other random variables and identically distributed as $R_v(t)$. By abuse of notation, we also denote this family of random variables as $R = \left(R_{v, e}(t)\right)_{(v, e) \in G, t \in \N}$.

A realisation of the matching model under an online assignment policy is as follow: first we assign the initial buffer $Z(0)$ into $X(0)$ in some manner.

At time $t$, a batch $k = k(t)$ is drawn i.i.d according to $\mu$. For $v \in V$, there is a batch of $k_v$ class-$v$ items arriving. The policy $\Phi$ chooses an assignment rate $\nu = \nu(t)$, and assign this batch to a matching type $e \ni v$ with probability $\frac{\nu_{v,e}}{\mu_e}$. The current buffer now can be denoted by $X'(t)$ where $X'_{v, e}(t) = X_{v, e}(t) + k_v$ if the batch of $k_v$ class-$v$ items is assigned to $e$, and $X'_{v, e}(t) = X_{v, e}(t)$ otherwise.

Then the reneging happens: for each $(v, e)$, amongst the class-$v$ items assigned to $e$, a percentage of $R_{v, e}(t)$ of them renege, leaving being only $B_{v, e}(t) = \left(1 - R_{v, e}(t)\right)X'_{v, e}(t)$ items.

Finally, for each matching type $e$, we form type-$e$ matchings using items assigned to $e$ in the order of arrivals, until no more matchings can be made. Formally, let $M_e(t) = \min_{v \in e} \llfloor \frac{B_{v, e}(t)}{e_v} \rrfloor$, then $X_{v, e}(t+1) = B_{v, e}(t) - e_v M_e(t)$.

\begin{remark}
	Note that we let the reneging happen \textit{after} the assignment but \textit{before} the matchings. Compared with our definition of the behaviour of a general policy, this is still consistent. In particular, for a general policy, the reneging happens only after the newly arriving items are admitted into the buffer.
	
	We could have let the reneging to happen before the assignment, but that would complicate the description and bring no further generality than what we have defined above.
\end{remark}

Similar to the non-reneging case, we say that $\Phi$ stabilises $(G, \mu, R)$ if there exists a compact set $C \subset \R_{\geq 0}^G$ such that for all $T \geq 0$, we have
\[
	\E\left[\min\{t \geq 0 \mid X(T + t) \in C\} \mid \mathcal{H}_T\right] < \infty.
\]
Note that this definition of stability is equivalent to the one using process $Z$ in the previous Subsection.

	\section{A case study: $Y$-topology}
\label{section: a case study}

In this Section, we study in depth the case of $Y$-topology as considered by Zubeldia et al. \cite{Zubeldia2026}. The stability is characterised \cite[Theorem 1]{Zubeldia2026}, but as a case study, it brings some fruitful intuition for the later analysis of more general cases.

Here, we consider a hypergraph $G$ of one single hyperedge $e$ containing three vertices, which we will denote by $u$, $v$, and $w$.
\begin{itemize}
	\item Amongst these classes, class-$u$ items do not renege, whilst class-$v$ and -$w$ items do, so in particular, we have $\E[R_u(t)] = 0$ and $0 < \E[R_v(t)] = \alpha_v, \E[R_w(t)] = \alpha_w \leq 1$.
	\item The matching model is unweighted, meaning $e_u = e_v = e_w = 1$, and only one item arrives at a time, so $\|\mu\|_1 = \mu_u + \mu_v + \mu_w = 1$.
\end{itemize}

An advantage of this example is that since there is only one hyperedge, there is no need to specify a matching policy, as we may, without loss of generality, match items whenever they form a matching.

\subsection{Stability characterisation}

Now let us consider how matchings occur. One way to see the formation of a matching is the following:
\begin{itemize}
	\item First, we divide it into two parts, the non-reneging part $e^n$ comprising one class-$u$ item, and the reneging part $e^r$ comprising one class-$v$ and one class-$w$ item.
	\item As some point, one class-$v$ and one class-$w$ item are in the buffer, which we combine to have the reneging part.
	\item Then, at some point, before either of the items in the reneging part reneges - and thus breaks the reneging part unless there are some other items to replace - an item of class-$u$ arrives, forms the non-reneging part on its own.
	\item These two parts are then combined together to form a matching.
\end{itemize}

Splitting the matching into two parts, we may see it as a double-ended queue with reneging on one side, or even an instance of a $M/M/1$ queue where the non-reneging part and the reneging part play the role of requests and services, respectively. Indeed, in the $M/M/1$ queue, the requests do not leave the buffer before they are served, whilst a surplus of services does no harm; in our case, the unmatched reneging parts will eventually renege on their own, thus do not affect the system's stability.

Regardless of seeing as a double-ended queue or as a $M/M/1$ model, the principle of stability remains the same: in order for the system to stabilise, there needs to be a surplus of services/reneging parts. In particular, during a given unit of time, the arrival rate $\lambda_e^n = \mu_u$ of the non-reneging parts must be strictly smaller than that $\lambda_e^r$ of the reneging parts.

But what is $\lambda_e^r$?

To find an appropriate definition of $\lambda_e^r$, let us revisit the $M/M/1$-queue point of view. Suppose we have a request, there are two possibilities: either a service arrives, serves the request, and clear the buffer, or another request arrives. If the latter occurs at least as often as the former, then the queue is unstable due to the accumulating requests. Or, in other words, we need that when there is a need for services, they arrive on average faster than requests.

Back to matching model, this tells us that when there is no reneging parts in the buffer, then the rate at which we have a new reneging part, is larger than $\mu_u$. This is the definition of $\lambda_e^r$. Another way to see this is that $\frac{1}{\lambda_e^r}$ is the average waiting time before we have a new reneging part assuming that there is no reneging parts in the buffer, formally written as 
\[
	\frac{1}{\lambda_e^r} = \E\left[\min\{t > 0 \mid Z_v(t), Z_w(t) \geq 1\} \mid Z_v(t) = 0 \vee Z_w(t) = 0\right].
\]

Once we have the definition, the proof of necessity and sufficiency of the condition $\lambda_e^r > \lambda_e^n$ follows.

\begin{lemma}
	\label{lemma: stability characterisation for Y-topology}
	The condition $\lambda_e^r > \lambda_e^n$ is necessary and sufficient for the stability of $G$.
\end{lemma}
\begin{proof}
	To prove the necessity, we formally express the intuition that without this condition, the class-$u$ items will accumulate.
	
	To this end, we consider a modified model $M'$ from the original matching model $M = (G, \mu, R)$: in $M'$, once a class-$v$ and a class-$w$ item are combined to form a reneging part, they will \emph{not} renege. In this case, $\lambda_e^r$ is truly the arrival rate of reneging parts.
	
	On the one hand, the number of reneging parts in $M$ is at most that in $M'$. On the other hand, in $M'$, the number of class-$u$ items in the buffer changes on average by the quantity $\lambda_e^n - \lambda_e^r$. If $\lambda_e^n \geq \lambda_e^r$, even in $M'$, we have $\E[\Delta Z_u(t)] \geq 0$ for all $t$, implying that $\E[\Delta Z_u(t)] \geq 0$ for all $t$ in $M$, so $Z_u$ is a submartingale.
	
	By stability, at time $\tau = \min \{t \geq 0 \mid Z(t) \in C\}$, we have $\E[Z_u(\tau)] \leq \max_{z \in C} z_u < \infty$ since $C$ is compact, implying that $Z_u(0) \leq \max_{z \in C}$, a contradiction since we can choose $Z_u(0)$ to be arbitrary.
	
	To prove the sufficiency, assume $\lambda_e^r > \lambda_e^n$, we consider the following Lyapunov function
	\[
		L(z) = C_0 z_u + \frac{z_v^2}{\alpha_v} + \frac{z_w^2}{\alpha_w^2},
	\]
	for some constant $C_0 > 0$ to be determined later. Let $Z(t) = z$, the drift of $L$ is given by
	\[
		\Delta L = C_0 \Delta z_u + \frac{1}{\alpha_v} \left[(\Delta z_v)^2 + 2z_v \Delta z_v\right] + \frac{1}{\alpha_w} \left[(\Delta z_w)^2 + 2z_w \Delta z_w\right].
	\]
	Denote by $k$ the arrival at time $t$, we have $\Delta z_v = k_v - R_v (z_v + k_v) - m_v$ where $m_v$ is the number of matching class-$v$ items at time $t$. In particular, we have $m_v(t) = 1$ if and only if there is a matching at time $t$ (note that since only one item arrives at a time, at most one matching is possible), and $m_v(t) = 0$ otherwise.
	
	Recall also that $k_v, R_v \leq 1$, we may bound $(\Delta z_v)^2$ trivially as
	\[
		(\Delta z_v)^2 = \left(k_v - R_v (z_v + k_v) - m_v\right)^2 \leq 9 + R_v^2 z^2,
	\]
	which gives
	\begin{align*}
		\Delta L 
		= & \left(1 + \frac{9}{\alpha_v} + \frac{9}{\alpha_w}\right) + C_0 \Delta z_u \\
		& + \frac{1}{\alpha_v} \left[R_v^2 z_v^2 + 2z_v \left(k_v - R_v (z_v + k_v) - m_v\right)\right]\\
		& + \frac{1}{\alpha_w} \left[R_w^2 z_w^2 + 2z_w \left(k_w - R_w (z_w + k_w) - m_w\right)\right] \\
		\leq & \left(1 + \frac{9}{\alpha_v} + \frac{9}{\alpha_w}\right) + C_0 \Delta z_u \\
		& + \frac{1}{\alpha_v} \left[R_v^2 z_v^2 + 2z_v \left(k_v - R_v (z_v + k_v)\right)\right]\\
		& + \frac{1}{\alpha_w} \left[R_w^2 z_w^2 + 2z_w \left(k_w - R_w (z_w + k_w)\right)\right] \\
		\leq & \left(1 + \frac{9}{\alpha_v} + \frac{9}{\alpha_w}\right) + C_0 \Delta z_u \\
		& - \frac{1}{\alpha_v} \left[R_v z_v^2 - 2 z_v k_v\right] - \frac{1}{\alpha_w} \left[R_w z_w^2 - 2 z_w k_w \right]
	\end{align*}
	where the last step is because $R_v, R_w \leq 1$.
	
	Taking expectation and completing the squares, we have
	\begin{align*}
		\E[\Delta L] 
		& \leq \left(1 + \frac{9}{\alpha_v} + \frac{9}{\alpha_w}\right) + C_0 \E[\Delta z_u] - \left(z_v^2 - 2z_v \frac{\mu_v}{\alpha_v}\right) - \left(z_w^2 - 2 z_w \frac{\mu_w}{\alpha_w}\right) \\
		& \leq \left(1 + \frac{9}{\alpha_v} + \frac{9}{\alpha_w} + \frac{\mu_v^2}{\alpha_v^2} + \frac{\mu_w^2}{\alpha_w^2}\right) + C_0 \E[\Delta z_u]\\
		& - \left(z_v - \frac{\mu_v}{\alpha_v}\right)^2 - \left(z_w - \frac{\mu_w}{\alpha_w}\right)^2
	\end{align*}
	
	Now let us fix some $\varepsilon > 0$.
	
	Note that as we match whenever possible, we have $\min \{z_u, z_v, z_w\} = 0$. If $z_u > 0$, then $\min \{z_v, z_w\} = 0$. By definition, at time $t + \tau$ where $\E[\tau] = \frac{1}{\lambda_e^r}$, there is a reneging part available; by Wald's equation, during this time, there will be $\frac{\lambda_e^n}{\lambda_e^r} < 1$ new class-$u$ items arriving, thus $\E[\Delta z_u]$ \emph{over the period $[t, t+\tau]$} is given by $\frac{\lambda_e^n}{\lambda_e^r} - 1 < 0$.
	
	Choosing
	\[
		C_0 > \frac{\varepsilon + \frac{1}{\mu_e^r}\left(1 + \frac{9}{\alpha_v} + \frac{9}{\alpha_w} + \frac{\mu_v^2}{\alpha_v^2} + \frac{\mu_w^2}{\alpha_w^2}\right)}{1 - \frac{\lambda_e^n}{\lambda_e^r}}, 
	\]
	we have 
	\begin{align*}
		& \E[L(Z(t + \tau)) - L(Z(t)) \mid Z_u(t) > 0] \\
		\leq & \E[\tau]\left(1 + \frac{9}{\alpha_v} + \frac{9}{\alpha_w} + \frac{\mu_v^2}{\alpha_v^2} + \frac{\mu_w^2}{\alpha_w^2}\right) \\
		& + C_0 \E[Z_u(t + \tau) - Z_u(t) \mid Z_u(t) > 0] \leq -\varepsilon.
	\end{align*}
	
	On the other hand, if $z_u = 0$, note that $\Delta z_u \leq 1$, so for 
	\[
		C_1 = C_0 +1 + \frac{9}{\alpha_v} + \frac{9}{\alpha_w} + \frac{\mu_v^2}{\alpha_v^2} + \frac{\mu_w^2}{\alpha_w^2},
	\]
	we have by Cauchy-Schwarz inequality, 
	\begin{align*}
		\E[\Delta L \mid Z_u(t) = 0] 
		& \leq C_1 - \left(z_v - \frac{\mu_v}{2\alpha_v}\right)^2 - \left(z_w - \frac{\mu_w}{2\alpha_w}\right)^2 \\
		& \leq C_1 - \frac{1}{2}\left(z_v + z_w - \frac{\mu_v}{2\alpha_v} - \frac{\mu_w}{2\alpha_w}\right)^2.
	\end{align*}
	If $z_u + z_v \geq C_2 = \frac{\mu_v}{2\alpha_v} + \frac{\mu_w}{2\alpha_w} + 2\sqrt{C_1 + \varepsilon}$, then $\E[\Delta L \mid Z_u(t) = 0] \leq -\varepsilon$.
	
	In all cases, we have that whenever $\|Z(t)\|_1 \geq C_2$, then we have $\E[L(Z(t + \tau)) - L(Z(t))] \leq -\varepsilon$ for some stopping time $\tau$ of finite mean. Per randomised multi-step Foster's criterion \cite[Theorem 2.1]{Yuksel2013}, the chain $Z$ is positive recurrent, which means that the system is stable.
\end{proof}

\subsection{Computing $\lambda_e^r$}

Lemma \ref{lemma: stability characterisation for Y-topology} gives a stability characterisation, but it is not entirely satisfactory since it relies on $\lambda_e^r$, which gives rise to the question: what is the relationship between $\lambda_e^r$, $\mu_v$, and $\mu_w$?

Let us first remark that this is a well-justified question: concretely, suppose we are given such a model $(G, \mu, R)$ with explicit value of $\mu$ and description of $R$, how do we verify the stability?

Ideally, we might hope for a simple relation computing $\lambda_e^r$. Moreover, as we have seen from the non-reneging case where the stability depends only on the first moment of $\mu$ (whilst assuming that $\mu$ is of finite second moment), we may expect the same for the reneging case: namely, there might be a stability criterion depending only on $\alpha$ and some mild assumptions on $R$.

It turns out that \emph{no such criterion exists}, as we will illustrate below with two examples.

\subsubsection*{Example 1: Independent reneging}

Suppose that each item, if reneging, does so independently. Per Remark \ref{remark: independent reneging}, this corresponds to having $R_v(t)$ follows the (normalised) binomial distribution $\frac{1}{Z_v(t) + k_v(t)} B(Z_v(t) + k_v(t), \alpha_v)$ with respect to the number $Z_v(t) + k_v(t)$ of class-$u$ items presented in the queue after the arrivals and the reneging rate $\alpha_v$, and similarly for class $w$.

To simplify this model even further, we revisit the original continuous-time model where the arrivals are according to independent Poisson processes, and the reneging time of, say, class-$v$ items to be exponentially distributed with parameter $\frac{1}{\alpha_v}$. In this case, with probability 1, no two items renege at the same time, so starting with this model, we may remove the binomial distribution for the reneging. The resulting discretised model is as follow: at time $t$,
\begin{itemize}
	\item an item of class $v$ (resp. $w$) reneges with probability $\alpha_v Z_v(t)$ (resp. $\alpha_w Z_w(t)$), or
	\item an item of class $u$ (resp. $v$, $w$) arrives with probability $\frac{\mu_u}{C}$ (resp. $\frac{\mu_v}{C}$, $\frac{\mu_w}{C}$).
\end{itemize}
where $C(t) = \alpha_v Z_v(t) + \alpha_w Z_w(t) + \|\mu\|_1$ is the normalisation constant.

Now let us compute $\lambda_e^r$. In what follows, we condition on $Z_v(t) = 0 \vee Z_w(t) = 0$. Consider the process $Q(t) = Z_v(t) - Z_w(t)$, we see that $Q$ is a birth-death process, where for $n \geq 0$, 
\[
\begin{gathered}
	p_n = \P\left[Q(t+1) = Q(t) + 1 \mid Q(t) = n \geq 0\right] = \frac{\mu_v}{C(t)}, \\
	q_n = \P\left[Q(t+1) = Q(t) - 1 \mid Q(t) = n > 0\right] = \frac{\mu_w + \alpha_v Q(t)}{C(t)}
\end{gathered}
\]
and for $n \leq 0$,
\[
\begin{gathered}
	p_n = \P\left[Q(t+1) = Q(t) + 1 \mid Q(t) = n < 0\right] = \frac{\mu_v - \alpha_w Q(t)}{C(t)}, \\
	q_n = \P\left[Q(t+1) = Q(t) - 1 \mid Q(t) = n \leq 0\right] = \frac{\mu_w}{C(t)}
\end{gathered}
\]
where with $C(t) = \begin{cases}
	1 + \alpha_v Q(t) & \text{ if } Q(t) \geq 0 \\
	1 -\alpha_w Q(t) & \text{ otherwise}
\end{cases}$.

Assuming the matching model $(G, \mu, R)$ is stable, $Z$ (and thus $Q$) is positive recurrent. $Q$ then admits the stationary distribution $\pi$ given by
\[
	\pi_n = \begin{cases}
		\pi_0 \prod_{k = 0}^{n-1} \frac{p_k}{q_{k+1}} & \text{ if } n \geq 0\\
		\pi_0 \prod_{k = n}^{-1} \frac{q_{k+1}}{p_k} & \text{ if } n < 0\\
	\end{cases},
\]
where $\pi_0^{-1} = \left(1 + \sum_{n = 1}^\infty \prod_{k = 0}^n \frac{p_k}{q_{k+1}} + \sum_{n = 1}^\infty \prod_{k = -n}^{-1} \frac{q_{k+1}}{p_k}\right)^{-1}$.

Conditioning on $Q(t) = n > 0$ (resp. $n < 0$), we have a new reneging part with probability $\frac{\mu_w}{\mu_w + \alpha_v Q(t)}$ (resp. $\frac{\mu_v}{\mu_v - \alpha_w C(t)}$). As such, we have
\[
	\lambda_e^r = \sum_{n > 0} \frac{\mu_w}{\mu_w + \alpha_v n} \pi_n + \sum_{n < 0} \frac{\mu_v}{\mu_v - \alpha_w n} \pi_n.
\]

\subsubsection*{Example 2: all-or-nothing}

At the other end of the spectrum compared with independent reneging is all-or-nothing reneging, namely at time $t$, with probability $\alpha_v$ (resp. $\alpha_w$), \emph{all} class-$v$ (resp. class-$w$) items renege.

In our setting, this means $R_v(t)$ (resp. $R_w(t)$) is independent and identically distributed for all $t$, following Bernoulli distribution of parameter $\alpha_v$ (resp. $\alpha_w$). Note that we still have $\E[R(t)] = \alpha$ as in the previous example.

On the other hand, we carry out the same analysis as above with the process $Q$, it is again positive recurrent. Suppose without loss of generality that $\mu_v \geq \mu_w$, considering the stationary distribution $\pi$ of $Q$ and for $n > 0$, the balance equation reads
\[
	\pi_n(\mu_v + \mu_w + \alpha_v n) = \pi_{n-1} \mu_v + \pi_{n+1} \mu_w,
\]
which gives a second-order linear recurrence sequence with affine coefficients, 
\[
	\pi_{n+1} = \frac{\mu_v + \mu_w + \alpha_v n}{\mu_w} \pi_n - \frac{\mu_v}{\mu_w} \pi_{n-1}.
\]
Let $\pi_n = \left(\frac{\mu_v}{\mu_w}\right)^\frac{n}{2} u_n$, we have
\[
	u_{n+1} = \frac{\mu_v + \mu_w + \alpha_v n}{\sqrt{\frac{\mu_v}{\mu_w}}} u_n - u_{n-1}.
\]
Set $\chi = \frac{2}{\alpha_v}\sqrt{\frac{\mu_v}{\mu_w}}$ and $\nu = \frac{\mu_v + \mu_w}{\alpha_v}$, a we have
\[
	u_{n+1} = \frac{2(n + \nu)}{\chi} u_n - u_{n-1},
\]
whence we recognise the recurrence relation of Bessel functions. Indeed, we have an explicit formula for $\pi_n$ in terms of Bessel functions of order $n + \nu$, given by
\[
	\pi_n = \left(\frac{\mu_v}{\mu_w}\right)^\frac{n}{2} \left[C_1 J_{n+\nu} \left(\frac{2}{\alpha_v}\sqrt{\frac{\mu_v}{\mu_w}}\right) + C_2 Y_{n+\nu} \left(\frac{2}{\alpha_v}\sqrt{\frac{\mu_v}{\mu_w}}\right)\right],
\]
for some suitable constants $C_1, C_2$. The second term with Bessel function of second kind is dominant, and in particular, as $n$ tends to infinity, we have
\[
	\left(\frac{\mu_v}{\mu_w}\right)^\frac{n}{2}\cdot Y_{n + \nu}\left(\alpha_v \sqrt{\frac{\mu_w}{\mu_v}}\right) \sim -\frac{1}{\pi} \Gamma(n + \nu) \alpha_v^n \to -\infty,
\]
so to maintain summability and positivity, we necessarily have $C_2 = 0$. A similar argument applies for $n < 0$, whence we may compute $\pi$ and eventually $\lambda_e^r$ explicitly.

\subsubsection*{Insight}

It should not be necessary to carry out the calculations in full in order to see that in general, $\lambda_e^r$ is different in both cases. These two examples illustrate that finding $\lambda_e^r$ requires knowing the exact nature of the reneging process $R$, and moreover, sometimes there might be no expression of $\lambda_e^r$ using elementary functions. Computing $\lambda_e^r$, and thus verifying stability, even in this simple case, is already a hard problem.

This is also the reason why we choose to have two reneging classes instead of one: if there were only one reneging class $v$, we always have by definition $\lambda_e^r = \mu_v$. As such, this also highlights the main difficulty between the reneging graph case and the reneging hypergraph case, which stems from the fact that, unlike the graph case, we need multiple items to obtain a reneging part, which requires more than one arrivals, during which other items may renege on their own.

As such, unlike the graph case where the stability criterion is a straightforward generalisation of the non-reneging case \cite{Jonckheere2023}, adding reneging in case of hypergraphs makes the problem considerably harder, and no analogues of \texttt{NCOND} condition exist.
	\section{Stability criterion in general case and necessity}
\label{section: necessity}

We have seen from the non-reneging case \cite{Nguyen2026, Nguyen2026a} that for hypergraphs, the crucial mechanism to stabilise a model is to manage the imbalance between the items assigned to a given matching type, and the necessary and sufficient criterion in this case essentially says that regardless of the buffer's state, the imbalance can always be mitigated \cite[Condition (1)]{Nguyen2026a}.

In the previous Section, we have a case study demonstrating another mechanism at play when there is reneging, namely the reneging parts must arrive more often than the non-reneging part. However, computing the rate of generating reneging parts is a difficult problem, where we have shown that, unlike the non-reneging case, stability is sensitive with respect to the reneging process $R$, depending on it as a whole and not just through its moments such as reneging rate $\alpha$. This begs the question: if verifying stability is already hard, is there any hope for a maximally stabilising policy?

The goal of this Section is two-fold:
\begin{itemize}
	\item We combine both mechanisms outlined above to give a stability criterion (Condition \eqref{eq: ncond}).
	\item We show that it is necessary for stability, and point out a technical difficulty in showing its sufficiency, which will be resolved in Section \ref{section: matching process} and Section \ref{section: sufficiency}.
\end{itemize}

\subsection{Stability criterion}

One advantage of the online assignment framework is that it allows to decouple the matching process and treat the hyperedges as independent. Roughly speaking, we may see $G$ approximately as $|E|$ copies of the case study in the previous Section.

For each hyperedge $e$, we split it into two parts, namely the non-reneging part $e^n = \{u \in e \mid \alpha_u = 0\}$ and the reneging part $e^r = \{v \in e \mid \alpha_v > 0\}$. The non-reneging parts collectively induce a matching sub-model $(G^n, \mu^n)$ where $G^n = (V^n, E^n)$ is a sub-hypergraph with $V^n = \{u \mid \alpha_u = 0\}$ the non-reneging classes and $E^n = \{e^n \mid e \in E\}$ the non-reneging parts of the hyperedges. Similarly, the reneging parts induces a sub-hypergraph $G^r = (V^r, E^r)$ with $V^r = \{u \mid \alpha_u > 0\}$ the non-reneging classes and $E^r = \{e^r \mid e \in E\}$ the non-reneging parts of the hyperedges.

If $(G, \mu, R)$ is stabilisable, then so is $(G^n, \mu)$; one may see it as a generalised instance of $(G, \mu, R)$ where the arrival rate of reneging classes is very large, so we always have reneging parts whenever necessary whilst having the stability unaffected. As such, we have the average matching rate $\lambda^n\in \R^{E^n}$, where $\lambda_e^n$ is the average number of type-$e^n$ matchings during a given unit of item.

The idea is to see $\mu_e^n$ as the ``arrival rate'' of the non-reneging class-$u$ items in the case study; the definition of $\mu_e^r$ remains the same, except that we generalise it for more than two reneging classes. Formally speaking, we define
\[
	\frac{1}{\lambda_e^r} = \E\left[\min\left\{t \biggr\vert \min_{v \in e^r} Z_{v, e}(t) - e_v \geq 0\right\} \biggr\vert \min_{v \in e^r} Z_{v, e}(t) - e_v < 0\right].
\]

Per the previous Section, we expect $\lambda_e^n < \lambda_e^r$; by the online assignment framework, the matching processes on hyperedges are decoupled, so we should expect this inequality to hold for all $e$. Of course, the precise value of $\lambda_e^n$ and $\lambda_e^r$ depends on our choice of assignment rate $\nu$ (recall that it is possible that $\lambda^n$ is not unique).

Our stability criterion essentially says that there exists an assignment rate such that this inequality always holds, in addition to the stability of $(G^n, \mu)$. There are many ways to formulate the latter, but for simplicity, we choose the linear algebraic characterisation \cite[Condition (3)]{Nguyen2026a}. In what follows, if all classes in a matching type $e$ are reneging, meaning $e^n = \emptyset$, then we define by convention that $\lambda_e^n = 0$. Similarly, if all classes in a matching type $e$ are reneging, meaning $e^r = \emptyset$, then we define by convention that $\lambda_e^r = \infty$.

\begin{equation}
	\label{eq: ncond}
	\textrm{\parbox{.8\textwidth}{The restriction $G\vert_{\R^{V^n}}$ of $G$ onto $\R^{V^n}$ is surjective. Moreover, there exists a positive solution $\lambda^n \in \R^{E^n}$ of the equation $G\vert_{\R^{V^n}} \lambda^n = \mu\vert_{\R^{V^n}}$, and there exists an assignment rate $\nu = (\nu_{v, e})_{v \in e^r, e \in E}$ of the reneging classes, such that for all $e \in E$, we have $\lambda_e^n \leq \lambda_e^r$.
	}}
\end{equation}

\subsection{Necessity of Condition \eqref{eq: ncond}}

The proof of necessity is essentially a generalisation of the law-of-large-number argument from the previous Section, except the interplay between hyperedges makes the proof more technical.

\begin{lemma}
	\label{lemma: necessity of ncond}
	Condition \eqref{eq: ncond} is necessary for the stability of $(G, \mu, R)$.
\end{lemma}
\begin{proof}
	First, we indeed have $(G^n, \mu)$ to be stabilisable: suppose $(G, \mu, R)$ is stabilisable, let $C$ be a compact set in the definition of stability, then it suffices to take the projection of $C$ onto $\R^{V^n}$, and we have the stability of $(G^n, \mu)$. As this is a non-reneging matching model, our previous results apply, which gives the surjectivity of $G\vert_{\R^{V^n}}$ and the existence of $\lambda^n$.
	
	It thus remains the existence of an assignment rate $\nu = (\nu_{v, e})_{v \in e^r, e \in E}$ of the reneging classes, such that for all $e \in E$, we have $\lambda_e^n \leq \lambda_e^r$.
	
	By our convention that $\lambda_e^r = \infty$ if $e^n = \emptyset$, without loss of generality, we may consider only the hyperedges $e$ with $e^n \neq \emptyset$. Furthermore, we may assume without loss of generality that $\mu_v > 0$ and $\alpha_v < 1$ for all $v$. As a consequence, the chain $Z$ when restricted to $\R^{V^r}$, where $V^r = \{v \mid \alpha_v > 0\}$, is ergodic even when the model is unweighted (so that it takes values in $\N^{V^r}$).
	
	On the other hand, note that the dependence of the transition rate on $\nu$ is linear: the probability of going from a buffer state $p$ to another buffer state $q$ can always be written as $\langle \ell_1, \nu \rangle + \langle \ell_2, \alpha \odot p \rangle$ where $\odot$ is the Hadamard product, and $\ell_1$ and $\ell_2$ are some fixed vectors depending on $p$, $q$, and the exact nature of the reneging process $R$.
	
	From the two paragraphs above, we conclude that if we define $Z'(t) \in \R^{G^r}$ as the process generating the reneging parts, that is
	\[
	\begin{gathered}
		Z''_{v, e}(t+1) = (1 - R_{v, e}(t))(Z'_{v, e}(t) + k_{v, e}(t)), \\
		Z'_{v, e}(t+1) = Z''_{v, e}(t) - e_v \min_{v \in e^r} \left\lfloor \frac{Z''_{v, e}(t)}{e_v} \right\rfloor,
	\end{gathered}
	\]
	then the map taking an assignment rate $\nu$ as arguments and giving the stationary distribution $\pi$ of $Z'$ is continuous with respect to total variation \cite{Liu2012}. As $\mu^r$ depends linearly (and thus continuously) on $\pi$ and the transition rate, the map taking an assignment rate $\nu$ as arguments and giving $\lambda^r$ is continuous.
	
	Now, we may proceed as in the proof of necessity for \texttt{NCOND}-like condition for the non-reneging case \cite[Lemma 3.1]{Nguyen2026}.
	
	On the one hand, the set of assignment rates $\nu$ is a convex compact polytope, so its image $\mathcal{P}^r$ under $\phi$ is also compact. On the other hand, one also sees that $\mathcal{P}^r$ is convex: consider two assignment rates $\nu_0$ and $\nu_1$ giving $\lambda_0^r$ and $\lambda_1^r$, then for $0 \leq x \leq 1$, taking the assignment rate $x \nu_1 + (1-x) \nu_2$ is equivalent to taking $\nu_1$ at each time $t$ with probability $x$, and $\nu_2$ with probability $1-x$. As such, the result $\lambda^r_x$ must satisfy that $\lambda_0^r \leq \lambda^r_x \leq \lambda_1^r$; this process is continuous, in the sense that as $x$ tends to $0$ (resp. $1$), then $\lambda^r_x$ tends to $\lambda_0^r$ (resp. $\lambda_1^r)$.
	
	On the other hand, the set of positive solutions $\lambda^n \in \R_{> 0}^{E^n}$ to the equation $G\vert_{\R^{V^n}} \lambda^n = \mu\vert_{\R^{V^n}}$ is a convex compact polytope $\mathcal{P}^n$: it is precisely the intersection between the affine subspace $\left(G\vert_{\R^{V^n}}\right)^{-1} \left(\mu\vert_{\R^{V^n}}\right)$ and the positive quadrant $\R_{> 0}^{E^n}$. As such, the product $\mathcal{P} = \mathcal{P}^r \times \mathcal{P}^n$ is convex and bounded.
	
	$\mathcal{P}$ lives in $\R^E \times \R^E$, on which we define the map $\Delta : \R^E \times \R^E \to \R^E$ given by
	\[
		\Delta (x, y) = x-y,
	\]
	which is linear, so the image $\mathcal{U} = \Delta \mathcal{P}$ of $\mathcal{P}$ under $\Delta$ is convex and closed.
	
	Suppose Condition \eqref{eq: ncond} fails, then this says that for all solutions $\lambda^n$ and assignment rates $\nu$, we have that for any point $x$ in $\mathcal{U}$, there exists a hyperedge $e \in E$ such that $x_e \geq 0$. As such, $\mathcal{U}$ and the negative quadrant $\R_{\leq 0}^E$ have disjoint interior. The set $\mathcal{U}$ is bounded, and $\R_{\leq 0}^E$ is closed, so by hyperplane separation theorem, there exists a hyperplane $H$ separating $\mathcal{U}$ and $\R_{\leq 0}^E$. We may without loss of generality translate $H$ so that it is tangent to the latter, but as it is a cone, $H$ will pass by the vertex, which is the origin, hence $H$ is now a linear hyperplane.
	
	Let $w$ be a normal vector of $H$, we can write $H = \{x \mid \langle x, w \rangle = 0\}$, and moreover, up to switching $w$ with $-w$ if necessary, we have
	\begin{itemize}
		\item $\langle x, w \rangle \geq 0$ for all $x \in \mathcal{U}$, and
		\item $\langle x, w \rangle \leq 0$ for all $x \in \R_{\leq 0}^E$.
	\end{itemize}
	The second condition says that $w$ is in the polar cone of $\R_{\leq 0}^E$, which is the positive quadrant $\R_{\geq 0}^E$, so $w$ has only non-negative coordinates.
	
	Finally, consider the function
	\[
		L_-(z) = \sum_{e \in E} w_e \left[\left(\min_{u \in e^n} \left\lfloor \frac{z_{u, e}}{e_u}\right\rfloor\right) - \left(\min_{v \in e^r} \left\lfloor \frac{z_{v, e}}{e_v}\right\rfloor\right)\right].
	\]
	
	The average long-term drift of $L_-(Z(t))$ is at least $\sum_{e \in E} w_e \left(\lambda_e^n - \lambda_e^r\right)$, which is non-negative by the first condition. On the one hand, the stability of $(G, \mu, R)$ implies that $L_-(Z(t))$ is uniformly bounded. On the other hand, we may set $X(0)$ so that $L_-(Z(0))$ is arbitrarily large, a contradiction.
\end{proof}

The proof of sufficiency is a modification of that from the previous Section, essentially a negative drift argument, except that we have an additional term to account for the stability of $(G^n, \mu)$.

Ideally, we would try to analyse the drift of some function resembling
\[
\begin{aligned}
	&\sum_{e \in E} \left[\frac{1}{\|e^n\|_1}\sum_{\{u, v\} \in (e^n)^2} e_u e_v\left(\frac{x_{u, e}}{e_u} - \frac{x_{v, e}}{e_v}\right)^2\right.\\
	& \left. +  C_0 \left(\left(\min_{u \in e^n} \left\lfloor \frac{x_{u, e}}{e_u}\right\rfloor\right) - \left(\min_{v \in e^r} \left\lfloor \frac{x_{v, e}}{e_v}\right\rfloor\right)\right) + \sum_{v \in e^r} \frac{x_{v, e}^2}{\alpha_v} \right].
\end{aligned}
\]
However, there is one last technical difficulty: unlike the term $\min_{v \in e^r} \left\lfloor \frac{x_{v, e}}{e_v}\right\rfloor$ whose (average long-term) drift is roughly given by $\mu_e^r$, the drift of the term $\min_{u \in e^n} \left\lfloor \frac{x_{u, e}}{e_u}\right\rfloor$ is rather hard to analyse.

Intuitively, in order to have a non-reneging part, we might need to wait more than one time step; when the waiting is over, since the arrivals are in batches, we might have more than one non-reneging parts. This irregularity makes it hard to control the drift of this latter term uniformly.

Note that this phenomenon does not occur in the case study since the creation of non-reneging parts involves only one class of items; nor does it occur in the non-reneging case since we \emph{ignore} entirely the creation of matchings and only care about the imbalance between the items assigned to the same matching type.
	\section{Matching process}
\label{section: matching process}

To address the problem of analysing the drift of $\min_{u \in e^n} \left\lfloor \frac{z_{u, e}}{e_u}\right\rfloor$ as we have identified, in this Section, we introduce a process called the matching process which acts as a proxy for this term, but whose dynamics is independent from the matching system and thus may be controlled more easily.

As this concerns only $(G^n, \mu)$ which is a non-reneging matching model, in this Section, we consider only the non-reneging case, meaning that we have $\alpha_v = 0$ for all $v$.

\subsection{Definition}

Recall that by the assumption that at any point, we make as many matchings as possible using items assigned to the same matching type, we have that the number of type-$e$ matchings $M_e(t)$ is given by $M_e(t) = \min_{v \in e} \llfloor \frac{B_{v, e}(t)}{e_v} \rrfloor$ and the dynamics of the process $M = (M_e(t))_{e \in E, t \in \N}$ is determined by $X$, the arrivals, and the reneging.

Recall also the intuition of our stability criterion using assignment rates \cite[Condition (1)]{Nguyen2026a}, where for all $G$-supports $U$, there exists an assignment rate $\nu$ such that for all $(v, e) \in U$, we have
\[
	\frac{\nu_{v, e}}{e_v} < \frac{\nu_e}{|e|}.
\]
The idea is that $U$ is the set of pairs $(v, e)$ where the class-$v$ items assigned to the matching type $e$ are in abundance, meaning we have $X_{v, e}(t) \geq e_v$. As all items are non-reneging, they may only leave the system by being matched, and the number of matchings required is heuristically $\frac{X_{v, e}(t)}{e_v}$. The change of this quantity is heuristically $\frac{\nu_{v, e}}{e_v} - \frac{\nu_e}{|e|}$; what this stability condition says is that we can always reduce the number of matchings needed, in some sense.

This intuition was made clearer when we showed \cite[Theorem 5.1]{Nguyen2026a} that $L$-MaxWeight policy is maximally stable where the function $L$ is given by
\[
L(X(t)) = \sum_{e \in E} \frac{1}{\|e\|_1} \sum_{\{u, v\} \in e^2} e_u e_v \left(\frac{X_{u, e}(t)}{e_u} - \frac{X_{v, e}(t)}{e_v}\right)^2.
\]
One way to interpret $L$ is that this function measures the imbalance amongst the items assigned to a given matching type: in particular, the larger it is, the more imbalance the items assigned, and the further away we are from making perfect matchings. By employing $L$-MaxWeight, we try to restore the balance, and the stability criterion \cite[Condition (1)]{Nguyen2026a} says that we always makes progress. Since no items leave the system without being matched, it suffices to nudge the buffer toward equilibrium however so slightly using some appropriate assignment rates.

This ``equilibrium'' has always been implicitly defined as the ideal states where $\frac{X_{u, e}(t)}{e_u} = \frac{X_{v, e}(t)}{e_v}$ for all $u, v \in e$. However, one way to make this notion explicit is by instead setting $\frac{X_{u, e}(t)}{e_u} = M_e(t)$ for all $u \in e$. In this sense, the ``imbalance'' is now measured by the differences $\frac{X_{u, e}(t)}{e_u} - M_e(t)$.

Now we remove the constraint $M_e(t) = \min_{v \in e} \llfloor \frac{B_{v, e}(t)}{e_v} \rrfloor$, and let $M$ be an $\R^E$-valued process on its own, which we call the matching process. We enlarge the system description to be the pair $(X, M)$.

The dynamics is given by the same equation as before: namely, after the arrivals, the buffer state is $B(t) = (B_{v, e}(t))_{(v, e) \in G}$. Then, we may define $M(t)$, whose difference with $M(t-1)$ may depend on the entire history $\mathcal{H}_{t-1}$ and the arrivals at time $t$ represented by $B(t)$. Finally, for $(v, e) \in G$, we let $X_{v, e}(t+1) = B_{v, e}(t) - e_v \min_{u \in e} \left\lfloor \frac{B_{u, e}(t)}{e_u}\right\rfloor$ and $M_e(t+1) = M_e(t) + \Delta M_e(t) - \min_{u \in e} \left\lfloor \frac{B_{u, e}(t)}{e_u}\right\rfloor$, where $\Delta M_e(t)$ is some deterministic increment at time $t$ (see the next Subsection).

\begin{remark}
	Readers familiar with literature on stochastic matching might recognise the similarity with the notion of virtual matchings by Nazari and Stolyar \cite{Nazari2019}. Indeed, one may interpret $M_e(t)$ as the number of virtual type-$e$ matchings.
	
	Other than distinct description and setting, compared with the framework of Nazari and Stolyar, there are three conceptual differences in ours:
	\begin{itemize}
		\item the evolution of $M$, or in other words, the decision to issue virtuals matchings or not, \emph{depends} on the arrival at time $t$;
		\item in our setting, virtual type-$e$ matchings may \emph{only} be realised by the items assigned to $e$; Nazari and Stolyar allowed such virtual matchings to be realised by \emph{any} items available in the buffer;
		\item strictly speaking, $M_e(t)$ is \emph{not} the number of virtual type-$e$ matchings, as we allow it to take non-integer values.
	\end{itemize}
\end{remark}

\subsection{A new stability criterion for the non-reneging case}

To see how to derive heuristically a stability criterion from this interpretation, consider a $G$-support $U$ comprised of pairs $(v, e)$ where the class-$v$ items assigned to the matching type $e$ are in imbalance. In this context, this is now defined as
\[
	U = \left\{(v, e) \biggr\vert \frac{X_{v, e}(t)}{e_v} \neq M_e(t)\right\}.
\]
Note that this set is a proper $G$-support in the sense of Definition \ref{definition: G-support}. Moreover, it induces two more $G$-supports as its partition, given by
\[
	U_+ = \left\{(v, e) \biggr\vert \frac{X_{v, e}(t)}{e_v} > M_e(t)\right\},
\]
and
\[
	U_- = \left\{(v, e) \biggr\vert \frac{X_{v, e}(t)}{e_v} < M_e(t)\right\}.
\]
One may see $U_+$ and $U_+$ are comprised of the pairs $(v, e)$ where the class-$v$ items assigned to the matching type $e$ are in shortage and in abundance, respectively.

Now for the increment $\Delta M(t)$, we choose a vector $\phi = (\phi_e)_{e \in E} \in \R^E$. A priori, this might depends on the arrivals and the assignments at time $t$, represented by $B(t) - X(t)$, but we will be more strict and require that $\Delta M$ depends only on the $G$-support $U$. Similarly, a priori $\Delta M$ is \emph{random}, depending on the arrivals and the assignments which are random themselves. However, we will be more strict and require that $\Delta M$ be \emph{deterministic}.

Thus, given an assignment rate $\nu$, the expected change of the quantity $\frac{X_{v, e}}{e_v} - M_e$ is given by
\[
	\E\left[\Delta\left(\frac{X_{v, e}}{e_v} - M_e\right) \biggr\vert X(t), M(t)\right] = \frac{\nu_{v, e}}{e_v} - \phi_e.
\]
If $(v, e) \in U_+$, we would expect this change to be negative, and similarly, if $(v, e) \in U_-$ (resp. $(v, e) \not\in U$), we would expect this quantity to be positive (resp. zero). This gives us a guess at a stability criterion.

\begin{equation}
	\label{eq:onlcond_non_reneging}
	\textrm{\parbox{.8\textwidth}{For all $G$-supports $U = U_+ \sqcup U_-$, there exist an assignment rate $\nu$ and $\phi$ depending only on $U$, such that for all $(v, e) \in U_+$,
	\[
		\frac{\nu_{v, e}}{e_v} < \phi_e,
	\]
	for all $(v, e) \in U_-$,
	\[
		\frac{\nu_{v, e}}{e_v} > \phi_e.
	\]
	and for all $(v, e) \not\in U$, 
	\[
		\frac{\nu_{v, e}}{e_v} = \phi_e.
	\]
	}}
\end{equation}

Note that there are many heuristic steps in the derivation of this condition. In particular, the assumption that $\Delta M$ depends only on $U$ and is deterministic, is not satisfied when $M_e(t) = \min_{v \in e} \llfloor \frac{B_{v, e}(t)}{e_v} \rrfloor$, so this condition is a priori stronger than what is necessary.

Fortunately, this strengthening brings no loss of generality. Indeed, Condition \eqref{eq:onlcond_non_reneging} is necessary for stability.

\begin{lemma}
	If $(G, \mu)$ is stabilisable, then Condition \eqref{eq:onlcond_non_reneging} is satisfied.
\end{lemma}
\begin{proof}
	If $(G, \mu)$ is stabilisable, ideally the matching process $M$ would capture the matchings themselves, in the sense that every matching is reflected in the change of $M$. As such, its increment $\phi_e$ can be interpreted as the matching rate $\lambda_e$, given by a solution to the conservation equation $G\lambda = \mu$.
	
	With this observation as our starting point, we use the linear algebraic characterisation of stability \cite[Condition (3)]{Nguyen2026a}, namely the stability of $(G, \mu)$ is equivalent to the surjectivity of $G$ and the existence of a positive solution $\lambda \in \R_{> 0}^E$ to the conservation equation $G\lambda = \mu$. Much of this proof is similar to that of the stability condition involving assignment rates alone \cite[Lemma 5.2]{Nguyen2026a}.
	
	Let $U$ be a $G$-support, partitioned into two smaller $G$-supports $U = U_+ \sqcup U_-$. For some $\varepsilon > 0$ to be defined later, let 
	\[
		\zeta_{u, e} = 
		\begin{cases}
			-\varepsilon & \text{ if } (u, e) \in U_+ \\
			0 & \text{ if } (u, e) \not\in U \\
			\varepsilon & \text{ if } (u, e) \in U_- \\
		\end{cases}.
	\]
	This induces a vector $\eta \in \R^V$ by $\eta_u = \sum_{e \ni u} \zeta_{u, e}$. Since $G$ is surjective, there exists $\xi \in \R^E$ such that $G \xi = \eta$.
	
	Now for $(u, e) \in G$, let $\nu_{u, e} = e_u (\lambda_e - \zeta_{u, e} + \xi_e)$ and $\phi_e = \lambda_e + \xi_e$, we will show that for suitably chosen $\varepsilon$, $\nu$ and $\phi$ satisfy Condition \eqref{eq:onlcond_non_reneging} for $U$.
	
	Already, by construction, we have that for all $u \in V$, we have
	\begin{align*}
		\sum_{e \ni u} \nu_{u, e}
		& = \sum_{e \ni u} e_u \left(\lambda_e + \zeta_{u, e} - \xi_e\right) \\
		& = \sum_{e \ni u} e_u\lambda_e + \sum_{e \ni u} e_u\zeta_{u, e} - \sum_{e \ni u} e_u\xi_e\\
		& = \mu_u + \eta_u - \eta_u = \mu_u.
	\end{align*}
	Moreover, we also have
	\[
		\frac{\nu_{u, e}}{e_u} - \phi_e = \lambda_e - \zeta_{u, e} + \xi_e - (\lambda_e + \xi_e) = \zeta_{u, e},
	\]
	so in particular, $\frac{\nu_{u, e}}{e_u} - \phi_e = -\varepsilon < 0$ if $(u, e) \in U_+$, $\frac{\nu_{u, e}}{e_u} - \phi_e = \varepsilon > 0$ if $(u, e) \in U_-$, and $\frac{\nu_{u, e}}{e_u} - \phi_e = 0$ if $(u, e) \not\in U$.
	
	Finally, note that $\|\eta\|_1 = |U|\varepsilon$. Since $G$ is a linear map, it is continuous, hence, for $\varepsilon > 0$ small enough, we have that $\|\xi\|_\infty \leq -\varepsilon + \min_{(u, e) \in G} \lambda_e$. This implies that for all $(u, e) \in G$.
	\[
		\nu_{u, e} = \lambda_e - \zeta_{u, e} + \xi_e \geq \lambda_e - \varepsilon - \|\xi\|_\infty \geq 0,
	\]
	thus $\nu$ is an assignment rate, as desired.
\end{proof}

It turns out, as expected, that Condition \eqref{eq:onlcond_non_reneging} is also sufficient, but the proof now relies on a different Lyapunov function whose negative drift is generated by this Condition.

Recall that the imbalance in this case is $\frac{X_{u, e}(t)}{e_u} - M_e(t)$. Naturally, as we did with the criterion involving only the assignment rates, to measure the total imbalance, it suffices to consider the sum of squares. Thus, our Lyapunov function is given by
\[
	L_+(X(t), M(t)) = \sum_{e \in E} \sum_{u \in e} \left(\frac{X_{u, e}(t)}{e_u} - M_e(t)\right)^2.
\]

\begin{lemma}
	\label{lemma: sufficiency of onlcond_non_reneging}
	If Condition \eqref{eq:onlcond_non_reneging} is satisfied, then $(G, \mu)$ is stabilisable.
\end{lemma}
\begin{proof}
	Our policy $\Phi$ is defined as usual in the proofs of sufficiency: we partition the initial buffer $Z(0)$ into $X(0)$ and choose $M(0)$ arbitrarily. Then, at time $t$, we choose an assignment rate $\nu$ and the change $\phi$ corresponding to the $G$-support $U$ of $X(t)$ given by Condition \eqref{eq:onlcond_non_reneging}, and proceed as an online assignment policy.
	
	To show that $\Phi$ stabilises $(G, \mu)$, consider the buffer $X(t) = x \in \R_{\geq 0}^G$ at time $t$. For $(u, e) \in G$, assume that $k_{u, e}$ class-$u$ items arrive and are assigned to $e$. We have that the change $\Delta \left[\left(\frac{X_{u, e}(t)}{e_u} - M_e(t)\right)^2\right]$ of the term $\left(\frac{X_{u, e}(t)}{e_u} - M_e(t)\right)^2$ is given by
	\begin{align*}
		\Delta \left[\left(\frac{X_{u, e}(t)}{e_u} - M_e(t)\right)^2\right]
		& = \left(\frac{k_{u, e}}{e_u} - \phi_e\right)^2 + 2\left(\frac{k_{u, e}}{e_u} - \phi_e\right)\left(\frac{x_{u, e}}{e_u} - m_e\right) \\
		& \leq \frac{k_{u, e}^2}{e_u^2} + \phi_e^2 + 2\left(\frac{k_{u, e}}{e_u} - \phi_e\right)\left(\frac{x_{u, e}}{e_u} - m_e\right).
	\end{align*}
	
	By linearity, we have
	\[
		\Delta L_+ \leq \sum_{e \in E} \sum_{u \in e} \frac{k_{u, e}^2}{e_u^2} + \phi_e^2 + 2\left(\frac{k_{u, e}}{e_u} - \phi_e\right)\left(\frac{x_{u, e}}{e_u} - m_e\right).
	\]
	Note that $0 \leq k_{u, e}$ and $\sum_{e \ni u} k_{u, e} = k_u$ by definition, where $k_u$ is the number of arriving class-$u$ items, implying that $\sum_{e \ni u} k_{u, e}^2 \leq \left(\sum_{e \ni u} k_{u, e}\right)^2 = k_u$. Denote $e_G = \min_{(u, e) \in G} e_u$ and taking the expectation over the arrivals and the assignments, we have
	\begin{align*}
		& \E[\Delta L_+ \mid X(t) = x, M(t) = m]\\
		\leq & \E\left[\sum_{e \in E} \sum_{u \in e} \frac{k_{u, e}^2}{e_u^2} + \phi_e^2 + 2\left(\frac{k_{u, e}}{e_u} - \phi_e\right)\left(\frac{x_{u, e}}{e_u} - m_e\right) \biggr\vert X(t) = x, M(t) = m\right] \\
		\leq & \E\left[\sum_{u \in e} \frac{k_u^2}{e_G^2} \biggr\vert X(t) = x, M(t) = m\right] \\
		& + \sum_{e \in E} \left(|e|\phi_e^2 + \sum_{u \in e} 2\left(\frac{\nu_{u, e}}{e_u} - \phi_e\right)\left(\frac{x_{u, e}}{e_u} - m_e\right)\right)\\
		\leq & \frac{\|\mu\|^2_2}{e_G^2} + \sum_{e \in E} \left(|e|\phi_e^2 + \sum_{u \in e} 2\left(\frac{\nu_{u, e}}{e_u} - \phi_e\right)\left(\frac{x_{u, e}}{e_u} - m_e\right)\right).
	\end{align*}
	By Condition \eqref{eq:onlcond_non_reneging}, we have 
	\[
		\left(\frac{\nu_{u, e}}{e_u} - \phi_e\right)\cdot \text{sgn}\left(\frac{x_{u, e}}{e_u} - m_e\right) < 0
	\]
	for $(u, e) \in U$. As such, considering all $(u, e) \in U$, we can choose $\varepsilon > 0$ so that 
	\[
		\left(\frac{\nu_{u, e}}{e_u} - \phi_e\right)\cdot \text{sgn}\left(\frac{x_{u, e}}{e_u} - m_e\right) \leq -\varepsilon.
	\]
	Moreover, since $G$ is finite, there are finitely many $G$-support $U$ and its partitions, so we may even choose $\varepsilon > 0$ to be uniform for all $U$, which means it depends only on $G$ and $\nu$.
	
	Plugging this back into $\E[\Delta L_+ \mid X(t) = x, M(t) = m]$, we obtain
	\[
		\E[\Delta L_+ \mid X(t) = x, M(t) = m] \leq \frac{\|\mu\|^2_2}{e_G^2} + \sum_{e \in E} |e|\phi_e^2 - 2\varepsilon\sum_{(u, e) \in G} \left|\frac{x_{u, e}}{e_u} - m_e\right|.
	\]
	
	Note that $|\phi_e| \leq \lambda_e + |\xi_e| \leq 2\lambda_e$, so let 
	\[
		C = \frac{\|\mu\|^2_2}{e_G^2} + 4\sum_{e \in E} |e|\lambda_e^2 \geq \frac{\|\mu\|^2_2}{e_G^2} + \sum_{e \in E} |e|\phi_e^2
	\]
	and $\delta > 0$ be arbitrarily chosen. By triangle inequality, we have a trivial bound
	\[
		\sum_{(u, e) \in G} \left|\frac{x_{u, e}}{e_u} - m_e\right| \geq \sum_{e \in E} \max_{u \in e} x_{u, e} - \min_{v \in e} \frac{x_{v, e}}{e_v}.
	\]
	Now note that by our definition of online assignment policies, we always have $\min_{v \in e} \frac{x_{v, e}}{e_v} < 1$. If $\|x\|_1 \geq \|G\|_1 \left(1 + \frac{C + \delta}{2\varepsilon}\right)$, then by pigeonhole principle, there exists $(u^*, e) \in G$ such that $\frac{x_{u^*, e}}{e_u} \geq 1 + \frac{C + \delta}{2\varepsilon}$, implying 
	\[
		\max_{u \in e} \frac{x_{u, e}}{e_u} - \min_{v \in e} \frac{x_{v, e}}{e_v} \geq \max_{u \in e} \frac{x_{u, e}}{e_u} - 1 \geq \frac{x_{u^*, e}}{e_u} - 1\geq \frac{C + \delta}{2\varepsilon},
	\]
	and hence
	\[
		\E[\Delta L_+ \mid X(t) = x, M(t) = m] \leq C - 2\varepsilon\sum_{(u, e) \in G} \left|\frac{x_{u, e}}{e_u} - m_e\right| \leq C - 2\varepsilon \frac{C + \delta}{2\varepsilon}.
	\]
	
	This means that for all $x \not\in C_\delta$ where
	\[
		C_\delta = \left\{x \biggr\vert \|x\|_1 <  \|G\|_1 \left(1 + \frac{C + \delta}{2\varepsilon}\right)\right\},
	\]
	we have $\E[\Delta L_+ \mid X(t) = x, M(t) = m] \leq -\delta$, implying for all $T \geq 0$, 
	\[
		\E[\min\{t \geq 0 \mid X(T + t) \in C_\delta\} \mid \mathcal{H}_T] \leq \frac{L_+(X(t), M(t))}{\delta},
	\]
	as desired.
\end{proof}

As such, we obtain a new characterisation of stability.

\begin{theorem}
	Let $(G, \mu)$ be a non-reneging matching model. Condition \eqref{eq:onlcond_non_reneging} is necessary and sufficient for its stability.	
\end{theorem}

\begin{remark}
	It should be pointed out that we could have established the necessity of Condition \eqref{eq:onlcond_non_reneging} without relying on the other criteria we had obtained in the previous paper. Indeed, it suffices to use similar arguments and establish the necessity within the class of online assignment policies, then introduce the reassignment and carry out analogous proofs, and finally pass by the linear algebraic characterisation of stability \cite[Condition (3)]{Nguyen2026a}. 
	
	Here, we choose not to do so for the sake of brevity. As the most essential part, the linear algebraic characterisation, was established, we may avoid all the technical arguments of the approach above.
\end{remark}
	\section{Sufficiency of Condition \eqref{eq: ncond} and MaxWeight policy}
\label{section: sufficiency}

\subsection{Sufficiency of Condition \eqref{eq: ncond}}

Armed with the notion of matching process, we are now ready to prove the sufficiency of Condition \eqref{eq: ncond}. The idea is to argue for the negative drift of some suitable Lyapunov function as outlined at the end of Section \ref{section: necessity}, where we will replace the term $\min_{u \in e^n} \left\lfloor \frac{x_{u, e}}{e_u}\right\rfloor$ with $M_e^n(t)$.

To highlight the relation between the matching process $M$ and the non-reneging part $(G^n, \mu)$ of the model, we will denote hereinafter the matching process with a superscript, as $M^n$.

\begin{lemma}
	\label{lemma: sufficiency of ncond}
	Condition \eqref{eq: ncond} is sufficient for the stability of $(G, \mu, R)$.
\end{lemma}
\begin{proof}
	For the sake of clarity, we will first describe a Markovian online assignment policy $\Phi$ which we will show to stabilise $(G, \mu, R)$.
	
	Consider the buffer state $X(t) = x$ and the matching process $M^n(t) = m^n$ at time $t$, let $U = \{(u, e) \mid u \in e^n \wedge x_{u, e} \geq e_u\}$ be the $G^n$-support corresponding to $x$ and $m$ given by
	\[
		U = \left\{(v, e) \biggr\vert \frac{X_{v, e}(t)}{e_v} \neq M_e^n(t)\right\},
	\]
	which is then partitioned into two smaller $G^n$-support $U = U_+ \sqcup U_-$ given by
	\[
		U_+ = \left\{(v, e) \biggr\vert \frac{X_{v, e}(t)}{e_v} > M_e^n(t)\right\},
	\]
	and
	\[
		U_- = \left\{(v, e) \biggr\vert \frac{X_{v, e}(t)}{e_v} < M_e^n(t)\right\}.
	\]
	
	For some $\varepsilon > 0$ to be defined later, let
	\[
		\zeta_{u, e} = 
		\begin{cases}
			-\varepsilon & \text{ if } (u, e) \in U_+ \\
			0 & \text{ if } (u, e) \in G^n \setminus U \\
			\varepsilon & \text{ if } (u, e) \in U_- \\
		\end{cases}.
	\]
	This induces a vector $\eta \in \R^{V^n}$ by $\eta_u = \sum_{e \ni u} \zeta_{u, e}$. Since $G^n$ is surjective, there exists $\xi \in \R^E$ such that $G^n \xi = \eta$.
	
	Fix a positive solution $\lambda^n$ of the conservation equation for $(G^n, \mu)$, namely $G^n \lambda^n = \mu\vert_{\R^{V^n}}$. For $(u, e) \in G^n$, let $\nu_{u, e} = e_u (\lambda_e^n - \zeta_{u, e} + \xi_e)$ and $\phi_e = \lambda_e^n + \xi_e$.
	
	To complete the assignment rate $\nu$, for $(v, e) \in G^r$, let $\nu_{v, e}$ be chosen so that $\lambda_e^r > \lambda_e^n$, whose existence is given by Condition \eqref{eq: ncond}.
	
	Finally, since $G^n$ is a linear map, it is continuous, hence, for all $\varepsilon > 0$ small enough, we have that $\|\xi\|_\infty \leq -\varepsilon + \min_{(u, e) \in G} \lambda^n_e$. This implies that for all $(u, e) \in G^n$.
	\[
		\nu_{u, e} = \lambda_e^n - \zeta_{u, e} + \xi_e \geq \lambda_e^n - \varepsilon - \|\xi\|_\infty \geq 0,
	\]
	thus $\nu$ is an assignment rate. Moreover, for all $\varepsilon > 0$ small enough, we also have $\|\xi\|_\infty \leq -\varepsilon + \lambda_e^r - \lambda_e^n$. We choose $\varepsilon > 0$ small enough so that $\nu$ is an assignment rate and that this condition is satisfied. This gives the assignment rate $\nu$ for the buffer state $X(t) = x$ and $M^n(t) = m^n$, thus the resulting policy $\Phi$ is Markovian, and the chain $(X, M)$ is a Markov chain.
	
	To show that $\Phi$ stabilises $(G, \mu, R)$, we first remark that for each hyperedge $e$ the assignment rate $\nu_{v, e}$ for $u \in e^r$ is chosen using only the $G^n$-support within that hyperedge; the choice of $\lambda^n$, $\varepsilon$ are all fixed and can be made uniform. Thus, the dynamics of each hyperedge is independent, and we decompose the buffer $X(t)$ into the hyperedge components $X^e(t) = (X_{u, e}(t))_{u \in e}$, so that $X(t) = (X^e(t))_{e \in E}$ (note that we use superscript notation to distinguish from the subscript notation $X_e(t) = \sum_{u \in e} X_{u, e}(t)$).
	
	Suppose that each of the hyperedge component is stable, meaning for each $e$ and $T \geq 0$, there exist a constant $c_e$ (possibly depending on T) and a compact set $C_e$ such that $\E[\tau^e_{C_e} \mid \mathcal{H}_T] \leq c_e$, where
	\[
		\tau^e_{C_e} = \min\{t \mid X^e(T + t) \in C_e\},
	\]
	then for $C = \prod_{e \in E} C_e$ and
	\[
		\tau_C = \min\{t \mid X(T + t) \in C\},
	\]
	as $X^e$'s evolve independently, we have $\tau_C = \min_{e \in E} \tau^e_{C_e} \leq \sum_{e \in E} \tau^e_{C_e}$, implying
	\[
		\E[\tau \mid \mathcal{H}_T] \leq \sum_{e \in E} \E[\tau^e_{C_e} \mid \mathcal{H}_T] \leq \sum_{e \in E} c_e < \infty.
	\]
	so $X$ will be stable. As such, it suffices to focus on a given hyperedge $e$ and show that $X^e$ is stable.
	
	The remaining of the proof is analogous to that of Lemma \ref{lemma: stability characterisation for Y-topology}. Namely, we fix a hyperedge $e$ and consider the function
	\[
		L_+^e(x, m^n) = \sum_{\{u, v\} \in (e^n)^2} \left(\frac{x_{u, e}}{e_u} - m_e^n\right)^2 + \sum_{v \in e^r} \frac{x_{v, e}^2}{\alpha_v} + C_0 m_e^n
	\]
	for some suitable constants $C_0 > 0$ to be specified later.
	
	For the first term
	\[
		L_1^e (x, m) = \sum_{\{u, v\} \in (e^n)^2} \left(\frac{x_{u, e}}{e_u} - m_e^n\right)^2,
	\]
	reusing calculations from the proof of Lemma \ref{lemma: sufficiency of onlcond_non_reneging}, we have
	\begin{align*}
		\E[\Delta L_1^e] 
		& \leq \frac{\|\mu_e\|^2_2}{e_G^2} + 4 |e|\lambda_e^2 - 2\varepsilon \sum_{u \in e^n} \left| \frac{x_{u, e}}{e_u} - m_e\right| \\
		& = C_1 - 2\varepsilon \sum_{(u, e) \in G^n} \left| \frac{x_{u, e}}{e_u} - m_e\right|.
	\end{align*}
	where $\|\mu_e\|^2_2 = \sum_{u \in e} \mu_u^2 = \int_{\R^V} k_e^2 \dd \mu(k)$ and $C_1 = \frac{\|\mu_e\|^2_2}{e_G^2} + 4 |e|\lambda_e^2 $ is an absolute constant depending only on $(G, \mu, R)$.
	
	For the second term
	\[
		L_2^e(x) = \sum_{v \in e^r} \frac{x_{v, e}^2}{\alpha_v},
	\]
	we repeat the calculation from the proof of Lemma \ref{lemma: stability characterisation for Y-topology}. For a given $e \in E$ and $v \in e^r$, we have
	\[
		\Delta x_{v, e}^2 = 2x_{v, e} \Delta x_{v, e} + (\Delta x_{v, e})^2.
	\]
	where $\Delta x_{v, e} = k_{v, e} - R_v (x_{v, e} + k_{v, e}) - e_v m_e$ with $m_e$ being the number of matching realised at time $t$, $k_{v, e}$ being the number of class-$v$ items newly arrived and assigned to $e$, and $R_v (x_{v, e} + k_{v, e})$ being the number of class-$v$ items assigned to $e$ and reneging at time $t$.
	
	Expanding $(\Delta x_{v, e})^2$, we have
	\[
		(\Delta x_{v, e})^2 = (k_{v, e} - R_v k_{v, e} - e_v m_e)^2 + R_v^2 x_{v, e}^2 - 2R_v x_{v, e} (k_{v, e} - R_v k_{v, e} - e_v m_e).
	\]
	We can bound $m_e$ by noticing that any matchings at time $t$ must involve items arriving at time $t$. In particular, for some $w \in e$, we have $x_{w, e} < e_w$. After the arrivals, and even without the reneging, we have at most $x_{w, e} + k_{w, e} < e_w + k_{w, e}$ items, thus there are at most $1 + \frac{k_{w, e}}{e_w} < 1 + \frac{k_{w, e}}{e_G}$ matchings, where $e_G = \min_{w \in e} e_w$.
	
	This gives a bound on the term $(k_{v, e} - R_v k_{v, e} - e_v m_e)^2$, simply by expanding
	\begin{align*}
		(k_{v, e} - R_v k_{v, e} - e_v m_e)^2 
		& \leq k_{v, e}^2 + e_v^2 m_e^2 \\
		& \leq k_{v, e}^2 + 2 e_v^2\left(1 + \frac{k^2_{w, e}}{e_G^2}\right) \\
		& \leq \|k\|^2 + 2 e_v^2\left(1 + \frac{\|k\|_2^2}{e_G^2}\right),
	\end{align*}
	for some $w$ such that $x_{w, e} < e_w$, where the first step is by noticing that $0 \leq R_v \leq 1$ and $k_{v, e}, m_e, e_v \geq 0$, the second step is by applying Cauchy-Schwarz inequality, and in the third step, we denote $\|k\|^2_2 = \sum_{(w, e) \in G} k_{w, e}^2$.
	
	Similarly, for the term	$2R_v x_{v, e} (k_{v, e} - R_v k_{v, e} - e_v m_e)$, one can bound
	\begin{align*}
		-2R_v x_{v, e} (k_{v, e} - R_v k_{v, e} - e_v m_e)
		& \leq 2R_v x_{v, e} (k_{v, e} + e_v m_e) \\
		& \leq 2R_v x_{v, e} \left(k_{v, e} + 1 + \frac{k_{w, e}}{e_G}\right) \\
		& \leq 2R_v x_{v, e} \left(\|k\|_1 + 1 + \frac{\|k_e\|_1}{e_G}\right),
	\end{align*}
	where $\|k\|_1 = \sum_{(w, e) \in G} k_{w, e}$.
	
	All in all, together with the fact that $0 \leq R_v \leq 1$, so that $R_v^2 \leq R_v$, taking expectation, we have
	\begin{align*}
		\E[(\Delta x_{v, e})^2]
		\leq & \|\mu\|_2^2 \left(1 + 2\frac{e_v^2}{e_G^2}\right) + 2e_v^2 + \alpha_v x_{v, e}^2 \\ 
		& + 2x_{v, e}\left[\alpha_v\|\mu\|_1 \left(1 + \frac{1}{e_G}\right) + \alpha_v\right].
	\end{align*}
	
	For the term $2x_{v, e} \Delta x_{v, e}$, by definition, we have
	\begin{align*}
		2x_{v, e} \Delta x_{v, e}
		& = 2 x_{v, e} [k_{v, e} - R_v (x_{v, e} + k_{v, e}) - e_v m_e] \\
		& \leq 2x_{v, e} \|k\|_1 - 2R_v x_{v, e}^2.
	\end{align*}	
	Taking expectation, we have $\E[2x_{v, e} \Delta x_{v, e}] \leq 2\|\mu\|_1 x_{v, e} - 2\alpha_v x_{v, e}^2$. Combining with the bound on $\E[(\Delta x_{v, e})^2]$, we have
	\begin{align*}
		\E[\Delta x_{v, e}^2]
		\leq & -\alpha_v x_{v, e}^2 + 2x_{v, e}\left[\|\mu\|_1 + \alpha_v\|\mu\|_1 \left(1 + \frac{1}{e_G}\right) + \alpha_v\right] \\
		& + \|\mu\|_2^2 \left(1 + 2\frac{e_v^2}{e_G^2}\right) + 2e_v^2 \\
		\leq & -\alpha_v \left[x_{v, e} - \|\mu\|_1 \left(1 + \frac{1}{\alpha_v} + \frac{1}{e_G}\right) - 1\right]^2 \\
		& + \|\mu\|_2^2 \left(1 + 2\frac{e_v^2}{e_G^2}\right) + 2e_v^2 + \alpha_v \left(\|\mu\|_1 \left(1 + \frac{1}{\alpha_v} + \frac{1}{e_G}\right) + 1\right)^2
	\end{align*}
	and by linearity
	\begin{align*}
		\E[\Delta L_2^e]
		= & \E\left[\sum_{v \in e^r} \frac{x_{v, e}^2}{\alpha_v}\right]  \leq \sum_{v \in e^r} \left[-\left(x_{v, e} - \|\mu\|_1 \left(1 + \frac{1}{\alpha_v} + \frac{1}{e_G}\right) - 1\right)^2 \right.\\
		& \left. + \frac{1}{\alpha_v} \left(\|\mu\|_2^2 \left(1 + 2\frac{e_v^2}{e_G^2}\right) + 2e_v^2 + \alpha_v \left(\|\mu\|_1 \left(1 + \frac{1}{\alpha_v} + \frac{1}{e_G}\right) + 1\right)^2\right)\right] \\
		= & C_2 - \sum_{v \in e^r} \left(x_{v, e} - \|\mu\|_1 \left(1 + \frac{1}{\alpha_v} + \frac{1}{e_G}\right) - 1\right)^2, \\
	\end{align*}
	where
	\[
		C_2 = \sum_{v \in e^r}\frac{1}{\alpha_v} \left[\|\mu\|_2^2 \left(1 + 2\frac{e_v^2}{e_G^2}\right) + 2e_v^2\right] + \left(\|\mu\|_1 \left(1 + \frac{1}{\alpha_v} + \frac{1}{e_G}\right) + 1\right)^2
	\]
	is an absolute constant depending only on $(G, \mu, R)$.
	
	In summary, we haves
	\begin{align*}
		\E[\Delta L_+^e]
		= & C_1 + C_2 - C_0\E[\Delta m_e^r] - 2\varepsilon \sum_{u \in e^n} \left| \frac{x_{u, e}}{e_u} - m_e\right| \\
		& - \sum_{v \in e^r} \left(x_{v, e} - \|\mu\|_1 \left(1 + \frac{1}{\alpha_v} + \frac{1}{e_G}\right) - 1\right)^2.
	\end{align*}
	
	Now we fix $\delta > 0$.
	
	Recall that $\min_{w \in e} x_{w, e} - e_u < 0$, so we have two cases.
	\begin{itemize}
		\item Either $\min_{v \in e^r} x_{v, e} - e_v < 0$. By definition, after some random time $\tau$ with mean $\E[\tau] = \frac{1}{\lambda_e^r}$, we have one reneging part. 
		
		Recall that we have $\|\xi\|_\infty \leq -\varepsilon + \lambda_e^r - \lambda_e^n$, so in particular, $\phi_e \leq \lambda_e^n + \| \xi\|_\infty \leq \lambda_e^r - \varepsilon$. During this time, excluding the possible matching happenings at time $t + \tau$, the increment $M_e^n(t + \tau) - M_e^n(t)$ is bounded by $\tau \cdot (\lambda_e^r - \varepsilon)$. Taking expectation, we have
		\[
			\E[M_e^n(t + \tau) - M_e^n(t) \mid \mathcal{H}_t] \leq \E[\tau] (\lambda_e^r - \varepsilon) = \frac{\lambda_e^r - \varepsilon}{\lambda_e^r} = 1 - \frac{\varepsilon}{\lambda_e^r},
		\]
		and with the (at least one) matchings happening at time $t + \tau$, we subtract this bound by $1$ and obtain 
		\[
			\E[L_2^e(X^et + \tau), M_e^n(t + \tau)) - L_2^e(X^et), M_e^n(t)) \mid \mathcal{H}_t] \leq \frac{C_1 + C_2 - C_0 \varepsilon}{\lambda_e^r}.
		\]
		
		Choosing $C_0 \geq \frac{C_1 + C_2 + \delta \lambda_e^r}{\varepsilon}$, we have that
		\[
			\E[L_+^e(X^et + \tau), M_e^n(t + \tau)) - L_+^e(X^et), M_e^n(t)) \mid \mathcal{H}_t] \leq -\delta.
		\] 
		
		\item Or $\min_{u \in e^n} x_{u, e} - e_u < 0$. Then, for all $x^e$ such that 
		\[
		\begin{aligned}
			\|x_e\|_1 = \sum_{w \in e} x_{w, e}
			& \geq C_3 = |e^r| \left[\|\mu\|_1 \left(1 + \frac{1}{\alpha_v} + \frac{1}{e_G}\right) + 1\right]\\
			& + \sqrt{|e^r| (C_1 + C_2 + \delta)} + |e^n| \left(\frac{C_1 + C_2 + \delta}{2\varepsilon}\right),
		\end{aligned}
		\]
		we have two subcases.
		\begin{itemize}
			\item Either $\|x^{e^n}\|_1 = \sum_{u \in e^n} x_{u, e} \geq |e^n| \left(\frac{C_1 + C_2 + \delta}{2\varepsilon}\right)$, in which case we argue as in the proof of Lemma \ref{lemma: sufficiency of onlcond_non_reneging}.
			
			In particular, by pigeonhole principle, there exists $u^* \in e^n$ such that $\frac{x_{u^*, e}}{e_{u^*}} \geq 1 + \frac{C_1 + C_2 + \delta}{2\varepsilon}$, implying
			\[
				\max_{u \in e^n} \frac{x_{u, e}}{e_u} - \min_{u' \in e^n} \frac{x_{u', e}}{e_{u'}} \geq \max_{u \in e^n} \frac{x_{u, e}}{e_u} - 1 \geq \frac{x_{u^*, e}}{e_{u^*}} - 1 \geq \frac{C_1 + C_2 + \delta}{2\varepsilon} 
			\]
			and hence
			\[
			\begin{aligned}
				\E[\Delta L_+^e \mid \mathcal{H}_t]
				& \leq C_1 + C_2 - 2\varepsilon \sum_{u \in e^n} \left| \frac{x_{u, e}}{e_u} - m_e\right| \\
				& \leq C_1 + C_2 - 2\varepsilon\frac{C_1 + C_2 + \delta}{2\varepsilon} = -\delta.
			\end{aligned}
			\]
			
			\item Or
			\[
			\begin{aligned}
				\|x^{e^r}\|_1 = \sum_{v \in e^r} x_{v, e} 
				\geq & |e^r| \left[\|\mu\|_1 \left(1 + \frac{1}{\alpha_v} + \frac{1}{e_G}\right) + 1\right]\\
				& + \sqrt{|e^r| (C_1 + C_2 + \delta)},
			\end{aligned}
			\]
			in which case, by Cauchy-Schwarz inequality, we have
			\[
			\begin{gathered}
				\sum_{v \in e^r} \left(x_{v, e} - \|\mu\|_1 \left(1 + \frac{1}{\alpha_v} + \frac{1}{e_G}\right) - 1\right)^2 \\
				\geq \frac{1}{|e^r|} \left[\|x^{e^r}\|_1 - |e^r|\left(\|\mu\|_1 \left(1 + \frac{1}{\alpha_v} + \frac{1}{e_G}\right) + 1\right)\right]^2 \\
				\geq \frac{1}{|e^r|} \left[\sqrt{|e^r| (C_1 + C_2 + \delta)}\right]^2 = C_1 + C_2 + \delta,
			\end{gathered} 
			\]
			implying
			\[
			\begin{aligned}
				\E[\Delta L_+^e \mid \mathcal{H}_t] 
				& \leq C_1 + C_2 - \sum_{v \in e^r} \left(x_{v, e} - \|\mu\|_1 \left(1 + \frac{1}{\alpha_v} + \frac{1}{e_G}\right) - 1\right)^2 \\
				& \leq C_1 + C_2 - (C_1 + C_2 + \delta) = -\delta.
			\end{aligned}
			\]
		\end{itemize}
	\end{itemize}
	In summary, after some random stopping time of mean at most $\frac{1}{\lambda_e^r} + 1$, we have the negative drift $-\delta$ for all $x^e \not \in C_e$ where,
	\[
		C_e = \{x^e \mid \|x^e\|_1 < C_3\},
	\]
	implying
	\[
		\E[\tau^e_{C_e} \mid \mathcal{H}_T] \leq c_e = \frac{L_+^e(X^e(T), M^e(T))}{\delta} \left(1 + \frac{1}{\lambda_e^r}\right) < \infty,
	\]
	as desired.
\end{proof}

Combining with Lemma \ref{lemma: necessity of ncond}, we have the main result of this article.
\begin{theorem}
	\label{theorem: necessity and sufficiency of ncond}
	Condition \eqref{eq: ncond} is necessary and sufficient for the stability of $(G, \mu, R)$. When $(G, \mu, R)$ is stabilisable, the stability can be achieved by a Markovian online assignment policy.
\end{theorem}

\subsection{Choice of $\varepsilon$ and MaxWeight policy}

The proof of Lemma \ref{lemma: sufficiency of ncond} gives a maximally stabilising policy of type randomised and online assignment. However, it relies on the existence of $\nu$, which is not necessarily constructive.

At the same time, the proof also shows that it suffices to guarantee negative drift for $L_+^e$ for all $e$, and a Max-Weight-type policy with respect to $(L_+^e)_{e \in E}$ will also suffice for this task. The only difficulty is to construct $L_+^e$, and in particular, to choose a suitable value for $C_0$.

The proof requires that $C_0 \geq \frac{C_1 + C_2 + \delta \lambda_e^r}{\varepsilon}$, which can be reduced to $C_0 > \frac{C_1 + C_2}{\varepsilon}$, since the argument works for all $\delta > 0$. For $\varepsilon > 0$, we must guarantee that $\|\xi\|_\infty \leq -\varepsilon + \min_{(u, e) \in G} \lambda^n_e$ and $\|\xi\|_\infty \leq -\varepsilon + \lambda_e^r - \lambda_e^n$. The first inequality can be easily guaranteed since everything is calculable; the second inequality constitutes the main obstacle, since we cannot compute $\lambda_e^r$ explicitly.

As such, it suffices to choose $\varepsilon$ sufficiently small, but there is no effective upper bound. One possible approaches to find suitable values for $\varepsilon$ is to choose an arbitrary initial value, run the $(L_+^e)$-MaxWeight policy where we try to minimise $\max_{e} L_+^e(X^e(t), M_e^n(t))$, approximate $\lambda_e^r$ empirically and adjust $\varepsilon$ accordingly.

The smaller $\varepsilon$, the larger $C_0$, which gives larger the bound on the return time to $C = \prod_{e \in E} C_e$. This is expected, as the policy is effectively more conservative in exchange for stability. This is comparable to the effect of parameter $\beta > 0$ in the average-reward asymptotically optimal policy by Nazari and Stolyar \cite[Section 5.B]{Nazari2019}, where a smaller choice of the parameter $\beta$ leads to more optimal average reward, in exchange for the policy being more conservative, the queue being longer, and thus the holding cost being higher.
	\section{General-weight matching}

Theorem \ref{theorem: necessity and sufficiency of ncond} concerns the stability of matching model $(G, \mu, r)$ with non-negative weights $G \in \R_{\geq 0}^{V \times E}$. Using the notion of generalised matchings we introduced in the previous article \cite{Nguyen2026a}, this result can be lifted to cover general-weight matching models. This constitutes the goal of this Section.

In the previous Sections, we have assumed that $\mu_v > 0$ for all $v \in V$, with the idea being that if $\mu_v = 0$, then there is no newly arriving items, so if class-$v$ items are used in some matchings, eventually we will run out of them, and thus those matching types will no longer be feasible. All the remaining matching types $e$ have $e_v = 0$, so eventually class-$v$ items will no longer be used. As such, without loss of generality, we can safely ignore class-$v$ items altogether.

However, in this Section, this is no longer the case: it may occur that whilst some matching types \emph{use} class-$v$ items, other matching types $e$ \emph{generate} class-$v$ items, meaning $e_v < 0$. One rather extreme example is a closed system where no new items arrive, but all matching types preserve the total number of items in the system, thus rendering it stable. As such, the fact that there are classes $v$ with $\mu_v = 0$ does not affect stability, nor should be neglected, and in this Section, we no longer assume that $\mu_v > 0$ for all $v \in V$.

\subsection{Generalised matchings}

Before extending Theorem \ref{theorem: necessity and sufficiency of ncond}, let us quickly summarise the notion of generalised matchings, which are essentially groups of matchings taken as a whole.

\begin{definition}
	A generalised matching type $m = \sum_{e \in E} m_e e$ is a $\N$-linear combination of matching types. For $u \in V$, we write $m_u = \sum_{e \in E} m_e e_u$, and $\|m\|_1 = \sum_{u \in V} |m_u|$.
	
	For a buffer state $B = (b_v)_{v \in V}$ after the arrivals and the reneging but before the matchings, a generalised matching type $m$ is feasible if we have $m_v \leq b_v$. If we choose to realise this matching type, then the resulting buffer is $X = (x_v)_{v \in V}$ with $x_v = b_v - m_v$.
	
	Such a matching type is said to be stabilising if it does not increase the content of the buffer for all classes, meaning $m_v \geq 0$ for all $v \in V$. In this case, we say that $m$ uses class-$v$ items and write $v \in m$ when $m_v > 0$. 
\end{definition}

As before, we omit the adjective ``generalised'' and ``stabilising'' for the sake of clarity whenever the context is clear.

A preliminary result says that for all classes $v \in V$ with incoming items, meaning $\mu_v > 0$, generalised stabilising matching types $m$ using class-$v$ items exist. The proof is almost word-by-word that from the non-reneging case \cite[Lemma 6.1]{Nguyen2026a}.

\begin{lemma}
	\label{lemma: existence of generalised matching types}
	Suppose $(G, \mu, R)$ is stabilisable, then for all $v \in V$ such that $\mu_v > 0$, there exist stabilising generalised matchings $m$ such that $m_v > 0$.
\end{lemma}
\begin{proof}
	Let $V' = \{ v \mid \mu_v > 0\}$. As $C$ is compact, it is bounded, so we can choose $M$ such that $M > x_v$ for all $x \in C$ and $v \in V'$.
	
	By definition, with probability $1$, for some $T$ large enough, we have that for all $v \in V'$, the total number of class-$v$ items arriving up to time $T$ is at least $M$. But if $(G, \mu)$ is stabilisable, we have that $\E[\tau \mid \mathcal{H}_T] < \infty$, where
	\[
		\tau = \min \{t \geq 0 \mid Z(t + T) \in C \}.
	\]
	So at time $T + \tau$, we have that $Z_v(T + \tau) < M$ by definition of $M$; we also have that $Z_v(T) \geq M$. This means that between $T$ and $T + \tau$, some simple matchings occurs. Combining all such matchings as a generalised matching $m$, we have that $m_v = Z_v(T) - Z_v(T + \tau) > 0$ for all $v \in V'$, as desired.
\end{proof}

There is one small caveat: consider such a matching type $m = \sum_{e \in E} m_e e$, it is entirely possible that there exists some class $u$ such that $m$ does not use class-$u$ items by definition, meaning that $m_u = 0$, but they \emph{are} used by some constituent matching types $e$, meaning for some $e \in E$, we have $m_e > 0$ and $e_u > 0$.

These used items are then eventually offset by newly created items from other matching types, but we nonetheless need a sufficient number of them presenting in the buffer in order for the constituent matchings $e$, and eventually the generalised matching $m$ to be realisable.

As such, we assume hereinafter that for all $v \in V$ such that $\mu_v = 0$, we have $X_v(0)$ sufficiently large, so that any generalised matching $m$ \emph{involving} (and not \emph{using}; of course, by the reasoning at the beginning of the Section, we necessarily have $m_v = 0$) class-$v$ items is realisable.

As a consequence, we may in fact ignore such classes $v$ altogether, and assume hereinafter that $\mu_v > 0$ for all $v$.

Equipped with this notion, let $M = \{m \mid m_v \geq 0\}$ be the set of generalised stabilising matchings, this induces a new matching hypergraph $G_M = (V, M)$, and thus a new matching model $(G_M, \mu, R)$. The stability of this model trivially implies that of the original $(G, \mu, R)$; conversely, per the proof of Lemma \ref{lemma: existence of generalised matching types}, we can translate the dynamics of $(G, \mu, R)$, if stabilisable, to that of $(G_M, \mu, R)$, so it suffices to study the stability of the latter.

\subsection{Reduction of $M$ and stability criterion}

It is tempting to apply Theorem \ref{theorem: necessity and sufficiency of ncond} to $(G_M, \mu, R)$, which is a matching model with non-negative weights, though there is only final caveat: that is, $M$ is \emph{not} finite, as any $\N$-multiple of a stabilising matching type remains stabilising.

And as we have seen \cite[Subsection 6.2]{Nguyen2026a}, when the weights are not in a cyclic $\Q$-submodule, meaning that they are not $\Q$-multiple of some fixed unit, then $M$ is not even finitely generated, and it is not possible to reduce $G_M$ to a finite hypergraph.

However, just as in the non-reneging case \cite[Section 7]{Nguyen2026a}, the idea is to realise that we do not need \emph{all} stabilising matchings, which constitutes the subject of this Subsection.

\begin{theorem}
	\label{theorem: necessity and sufficiency of ncond for general-weight case}
	$(G, \mu, R)$ is stabilisable if and only if it is stabilisable using finitely many stabilising matching types. When this is the case, one may need exactly $|E|$ stabilising matching types, and this is optimal.
	
	In view of Theorem \ref{theorem: necessity and sufficiency of ncond}, this is equivalent to the existence of a matrix $N \in \N^{E \times E}$ such that the map $G^n_N = G^n \circ N$ is surjective, that the equation $G^n_N \lambda^n = \mu\vert_{\R^{V^n}}$ admits a positive solution $\lambda^n \in \R^E_{> 0}$, and there exists an assignment rate $(\nu_{u, e})_{(u, e) \in G^r}$ such that for all $e \in E$, $\lambda_e^n < \lambda_e^r$.
\end{theorem}
\begin{proof}
	As we have seen, the stability of $(G, \mu, R)$ implies that of $(G^n, \mu)$. Using the result from the non-reneging case \cite[Theorem 7.1]{Nguyen2026a}, this is equivalent to the fact that it is stabilisable using finitely many stabilising matching types.
	
	This is also equivalent to the existence of the matrix $N$ such that the map $G^n_N = G^n \circ N$ is surjective, that the equation $G^n_N \lambda^n = \mu\vert_{\R^{V^n}}$ admits a positive solution $\lambda^n \in \R^{E^n}_{> 0}$. As an abuse of notation, we also denote by $N$ a finite set of stabilising matching types necessary to stabilise $(G^n, \mu)$. We have thus reduce $G^n$ to $G^n_N = (V^n, N)$, which in turns induces a matching hypergraph $G_N = (V, N)$.
	
	The stability of $(G, \mu, R)$ implies the stability of $(G_N, \mu, R)$, which in turns gives the existence of $\nu$ such that for all $e \in E$, $\lambda_e^n < \lambda_e^r$. Conversely, given $\nu$ and $N$, we have that $(G_N, \mu, R)$ is stabilisable, which implies the same for $(G, \mu, R)$.
\end{proof}

\begin{remark}
	It seems possible to carry out a direct proof analogous to the non-reneging case, but this would seem arduous. The idea of the proof for the non-reneging case is to consider the set $C_V$ of arrival rates $\mu$ such that $(G, \mu)$ is stabilisable, and its subset $F_V$ of arrival rates $\mu$ such that $(G, \mu)$ is stabilisable using finitely many stabilising matching types. Then, using some topological arguments, we show that $F_V$ is close and dense in $C_V$, thus equals $C_V$ itself.
	
	The argument for closeness and density involves using linear algebraic characterisation of stability. In the reneging case, we would also have to ensure the existence of an assignment rate $\nu$ such that for all $e \in E$, we have $\lambda_e^n < \lambda_e^r$, which is a complicated task.
\end{remark}
	\section{Conclusion}
\label{section: conclusion}

We have incorporated reneging into stochastic matching models on hypergraphs. The interplay between reneging classes when there are two or more of them in one matching type, which does not occur when we restrict ourselves to graphs, makes generalising the result for the graph case by Jonckheere et al. \cite{Jonckheere2023} using the online assignment framework a non-trivial task.

From a case study, we have shown, to our best knowledge, the first sensitivity result in stochastic matching, namely that the stability depends on the exact nature of the reneging process $R$, and not only on its moments. This is in stark contrast with the non-reneging case \cite{Nguyen2026a}.

At the same time, we have uncovered a new stabilising mechanism at play, which, together with the online assignment framework, gives us a stability criterion (Condition \eqref{eq: ncond}). Whilst the proof of necessity is a generalisation of the law-of-large-number argument, to show that it is sufficient, we introduced the notion of matching process, which acts as a proxy for the number of non-reneging parts in a matching type, but is more amenable to analysis.

Once it is developed, the rest of Lyapunov argument from the case study carries over (Theorem \ref{theorem: necessity and sufficiency of ncond}). As a corollary, the proof gives a family of MaxWeight-type policies parameterised by $\varepsilon > 0$, which are maximally stabilising for all $\varepsilon$ sufficiently small. This mirrors the parameter $\beta$ in the policy family introduced by Nazari and Stolyar, in the sense that a smaller choice may ensure stability but at the cost of larger average queue length (and, per Little's law, longer average waiting time for the non-reneging items). Unfortunately, there is no effective way to choose a good value for $\varepsilon$ a priori, but there are ways to adjust it as the policy is being implemented.

Using what we had known for the non-reneging case \cite[Section 7]{Nguyen2026a}, and notably the notion of generalised matching types \cite[Section 6]{Nguyen2026a}, we generalised our stability result from the non-negative-weight case to the arbitrary-weight case (Theorem \ref{theorem: necessity and sufficiency of ncond for general-weight case}).

It is possible to generalise this result further for continuous-time matching models, as we did in the non-reneging case \cite[Section 8]{Nguyen2026a}. We may expect that a periodic-review MaxWeight-type policy is also maximally stabilising, but we choose not to include in the article, for the reason that the development is minimal.
	
	\bibliographystyle{ieeetr}
	\bibliography{refs}
\end{document}